\documentclass{article}
\usepackage[english]{babel}
\usepackage[utf8]{inputenc}
\usepackage{tabularx}
\usepackage{etoolbox}
\usepackage{layout}
\usepackage{multirow}
\usepackage{subcaption}
\usepackage{afterpage}
\usepackage{amsmath}
\usepackage{amsthm}
\usepackage{amssymb}
\usepackage{enumitem}
\usepackage{float}
\usepackage{amsfonts}
\usepackage[nobysame]{amsrefs}
\usepackage{bbm}
\usepackage[font={sf, small},skip=0 cm]{caption}
\usepackage{graphicx}
\usepackage{mathtools}
\usepackage{fancyhdr}
\usepackage{xcolor}
\usepackage{listings}
\usepackage{comment}
\usepackage[mathscr]{euscript}
\usepackage[toc,page]{appendix}
\usepackage[top=2.2cm, bottom =2.2 cm,left=3.5cm,right=3.5cm,includeheadfoot]{geometry}
\usepackage{titlesec}
\titleformat{\subsection}{\large\bfseries}{\thesubsection}{1em}{}
\usepackage{amsthm}

\theoremstyle{plain}
\newtheorem{thm}{Theorem}[section]
\theoremstyle{plain}
\newtheorem{lemma}[thm]{Lemma}

\theoremstyle{plain}
\newtheorem*{thm*}{Theorem}
\theoremstyle{plain}
\newtheorem{prop}[thm]{Proposition}
\theoremstyle{plain}
\newtheorem{corollary}[thm]{Corollary}
\theoremstyle{definition}
\newtheorem{definition}[thm]{Definition}
\theoremstyle{remark}
\newtheorem{remark}{Remark}

\DeclareMathOperator{\totdertime}{\frac{d}{dt}}

\DeclareMathOperator{\differ}{D}

\DeclareMathOperator{\diver}{div}
\DeclareMathOperator{\rotore}{curl}

\DeclareMathOperator{\identity}{I}
\DeclareMathOperator{\lapl}{\Delta}

\DeclareMathOperator{\dist}{\mathbf{d}}
\DeclareMathOperator{\idnt}{I}

\usepackage{animate}
\usepackage{float}
\usepackage{enumitem}
\usepackage{authblk}

\title{\centering Strong convergence of the vorticity and conservation\\of the energy for the $\alpha$-Euler equations}
\author[1]{Stefano Abbate}
\author[2]{Gianluca Crippa}
\author[3]{Stefano Spirito}

\affil[1]{\small Gran Sasso Science Institute (GSSI), Viale Francesco Crispi 7, I-67100, L'Aquila, Italy}
\affil[2]{\small Departement Mathematik und Informatik, Universität Basel, Spiegelgasse 1, CH-4051, Basel, Switzerland}
\affil[3]{\small DISIM - Dipartimento di Ingegneria e Scienze dell'Informazione e Matematica, Università degli Studi de L'Aquila, Via Vetoio, I-67100, L'Aquila, Italy}

\date{}

\begin{document}

\maketitle

\begin{abstract}
    In this paper, we study the convergence of solutions of the $\alpha$-Euler equations to solutions of the Euler equations on the $2$-dimensional torus. In particular, given an initial vorticity $\omega_0$ in $L^p_x$ for $p \in (1,\infty)$, we prove strong convergence in $L^\infty_tL^p_x$ of the vorticities $q^\alpha$, solutions of the $\alpha$-Euler equations, towards a Lagrangian and energy-conserving solution of the Euler equations. Furthermore, if we consider solutions with bounded initial vorticity, we prove a quantitative rate of convergence of $q^\alpha$ to $\omega$ in $L^p$, for $p \in (1, \infty)$. 
\end{abstract}

\section{Introduction}
\let\thefootnote\relax\footnote{{\em2023 Mathematics Subject Classification.} Primary: 76B03, 35Q35. Secondary: 35Q31, 76F65.
	
\hspace{0.2 cm}{\em Key words and phrases.} 2D Euler equations, 2D alpha-Euler equations, Lagrangian solutions, conservation of the energy.
}
In this paper, we consider the incompressible $\alpha$-Euler equations on the two-dimensional torus, which, given $\alpha>0$, read as
\begin{flalign}
	\begin{cases}
		\partial_t v^\alpha +u^\alpha\cdot \nabla v^\alpha+\sum_j v^\alpha_j \nabla u^\alpha_j = -\nabla p ,&\quad \text{on}  \quad(0,T) \times \mathbb{T}^2 \\
		v^\alpha := u^\alpha - \alpha \lapl u^\alpha ,&\quad \text{on} \quad(0,T) \times \mathbb{T}^2 \\
		\diver u^\alpha = \diver v^\alpha = 0,&\quad \text{on}\quad (0,T) \times \mathbb{T}^2 \\
		u^\alpha(0, \cdot) = u^\alpha_0, &\quad \text{on} \quad\mathbb{T}^2
	\end{cases}\label{eq:alphaeuler}
\end{flalign}
and our primary objective is the rigorous study of the limit as $\alpha\to 0$ of solutions of \eqref{eq:alphaeuler}. Formally, if we substitute $\alpha= 0$ in \eqref{eq:alphaeuler}, we obtain that $v^\alpha=u^\alpha$, and then, employing the identity $$\sum_j u^\alpha_j \nabla u^\alpha_j= \frac{\nabla |u^\alpha|^2}{2}$$ and defining 
\begin{math}
	\pi^\alpha:= p+|u^\alpha|^2/2,
\end{math} 
we obtain the two-dimensional incompressible Euler equations:
\begin{equation}\label{eq:eulerintro}
\begin{cases}
\partial_t u+(u\cdot\nabla)u=-\nabla \pi, &\quad \text{on} \quad(0,T) \times \mathbb{T}^2 \\
\diver u=0, &\quad \text{on} \quad(0,T) \times \mathbb{T}^2\\
u(0, \cdot)=u_0, &\quad \text{on} \quad \mathbb{T}^2.
\end{cases}
\end{equation}

The $\alpha$-Euler equations are part of a larger class of approximation schemes called {\em Large Eddies Simulations (LES)}, which have been first introduced in \cite{smago} by Smagorinsky and later generalized by Leonard \cite{leonard}. LES approximations are relevant for numerical simulations of fluids in turbulent regime.   Indeed, due to the high number of scales needed in turbulent dynamics,{ \em Direct Numerical Simulation (DNS) }of fluid equations are computationally very demanding. The idea of LES models is based on the fact that the whole range of flow scales may not be necessary in order to have an accurate approximation. Therefore, a filter {\em cutting the small scales} is applied to the velocity. Since the operation of filtering does not commute with the nonlinear convective term, one needs to estimate the commutator between the filter and the convective term and this procedure gives rise to an approximation. We refer to \cite{bookberselli} and \cite{gop} for more details on the derivation and the motivations of LES models. In particular, we refer to \cites{fht, mars} for the derivation of the $\alpha$-Euler equations. 

Concerning the system \eqref{eq:alphaeuler}, the filter is given by the so-called Helmholtz filter, which is exactly the second equation in \eqref{eq:alphaeuler}, namely
\begin{equation}
	u^\alpha:= (\idnt-\alpha \lapl)^{-1} v^\alpha. \label{eq:helmotz}
\end{equation} 
The action of the filter \eqref{eq:helmotz} can be written in Fourier variables as
\begin{equation}
	\widehat{u^\alpha}(k) =\frac{	\widehat{v^\alpha}(k)}{1+\alpha|k|^2}, \qquad \forall k \in \mathbb{Z}^2.\label{eq:fourierhelm}
\end{equation}
The denominator of the right hand side of \eqref{eq:fourierhelm} diverges as the frequency grows, {\em cutting} the high frequencies as a consequence. 

The analysis of the limit as $\alpha\to0$ for solutions of the $\alpha$-Euler equations has been widely studied in literature. In particular, we mention the important result in \cites{lopestitizang}, where the authors analyse the convergence in $L^2$ for the velocity fields in two dimensions for smooth solutions of the \eqref{eq:alphaeuler} equations on a bounded domain with Dirichlet boundary conditions towards smooth solutions of the Euler equations. The importance relies on the fact that in the limit of $\alpha\to 0$ no boundary layers are created for the velocity. The result in \cites{lopestitizang} has been extended in \cite{bi} where the convergence of the velocity for the problem posed on bounded domains with Dirichlet boundary conditions has been proved for less regular solutions. We also recall the paper  \cite{bill}, where the convergence as $\alpha\to 0$ is studied for initial vorticity in the space of positive Radon measures. Finally, we refer to \cites{olishk,holm} where the relationship between the $\alpha$-Euler equations and the vortex blob method (another commonly used numerical approximation of the Euler equations) has been investigated. 

An important feature of the system \eqref{eq:alphaeuler} is that the vorticity structure of the $2D$ Euler equations is preserved in the approximation. Indeed, thanks to the presence of the term $\sum_j v^\alpha_j \nabla u^\alpha_j$ in the momentum equations of \eqref{eq:alphaeuler}, if we consider $q^{\alpha}=\rotore v^{\alpha}$, we obtain the following vorticity formulation of \eqref{eq:alphaeuler} 
\begin{flalign}
	\begin{cases}
		\partial_t q^\alpha +u^\alpha\cdot \nabla q^\alpha =0,&\quad \text{on}  \quad(0,T) \times \mathbb{T}^2 \\
		q(0, \cdot) = q_0^\alpha,&\quad \text{on} \quad\mathbb{T}^2\\
		\diver u^\alpha=\diver v^\alpha=0,&\quad \text{on}  \quad(0,T) \times \mathbb{T}^2 \\
		\rotore v^\alpha= q^\alpha,&\quad \text{on}  \quad(0,T) \times \mathbb{T}^2.
	\end{cases}\label{eq:alphaeulervort}
\end{flalign}

The vorticity formulation \eqref{eq:alphaeulervort} is particularly important for the purpose of this paper, since we are primarily interested in the analysis of the convergence of the vorticity $q^{\alpha}$ towards the vorticity of the $2D$ Euler equations in {\em strong} norms. This is reminiscent of the analogous results for the vanishing viscosity limit. In two-dimensional turbulence, the transport-structure for the vorticity and the appearance of an inverse cascade (from small to large scales) entail that vanishing viscosity solutions enjoy better properties compared to general weak solutions. In this paper, we show that the same principle holds for solutions that are the limit of the $\alpha$-Euler equations. 

In our first result, we show the strong convergence  in $L^p$ of the vorticity to a solution of the Euler equations which conserves the energy and  is Lagrangian, namely the vorticity is transported by the flow of the associated velocity, see \textsc{Definition \ref{def:lagsol}}. We refer to \textsc{Section \ref{sec:prelim}} for the relevant notations and definitions. In particular, $k$ is the Biot-Savart kernel on the torus (cf. \eqref{eq:biotsavart}).
\begin{thm}\label{thm:strongconvvort}
	Let $T>0$ arbitrary and finite, $p \in (1, \infty)$ and $\omega_0 \in L^p(\mathbb{T}^2)$ with $\int_{\mathbb{T}^2} \omega_0 = 0$. Let $\{q_0^\alpha\}$ be a sequence of functions uniformly bounded with respect to $\alpha$ in  $L^p(\mathbb{T}^2)$ with $\int_{\mathbb{T}^2} q_0^\alpha = 0$ such that
	\begin{equation*}
		q_0^\alpha \rightarrow \omega_0 \quad \text{strongly in} \enskip L^p (\mathbb{T}^2).
	\end{equation*}
	Let $(u^\alpha, q^\alpha)$ be the solution of the $\alpha$-Euler equations with initial datum $q_0^\alpha$. Then, up to subsequences, there holds, 
	\begin{equation*}
		u^\alpha \rightarrow u \quad \text{strongly in} \enskip \mathcal{C}([0,T];L^2 (\mathbb{T}^2))
	\end{equation*}
	and
	\begin{equation*}
		q^\alpha \rightarrow \omega \quad \text{strongly in} \enskip \mathcal{C}([0,T];L^p (\mathbb{T}^2)),
	\end{equation*}
and $(u, \omega)$ is a Lagrangian solution of the Euler equations.
	Moreover, for any $\delta >0$, there exists $K(\delta, \omega_0)$ such that for $\alpha$ small enough it holds
	\begin{equation}
		\sup_{t \in [0,T]}\|q^\alpha(t) -\omega(t) \|_{L^p} \leq \delta +\frac{K(\delta, \omega_0)}{\left| \log\left(\|u^\alpha-u\|_{L^1_tL^1_x}\right)\right|}+\|q^\alpha_0-\omega_0\|_{L^p}.\label{eq:stimanonquant}
	\end{equation}
	Finally, the solution $u$ conserves the kinetic energy, namely
	\begin{equation}
		\|u(t)\|_{L^2} = \|u_0\|_{L^2}, \quad \forall t \in (0,T), \quad \text{where} \quad u_0 = k \ast \omega_0.\label{eq:consenergyfinal}
	\end{equation}
\end{thm}
\bigskip 
The novelty of our approach consists in the application of techniques related to the Lagrangian perspective, introduced in the non-smooth settings in \cite{odeest} and employed for a vanishing viscosity scheme in \cite{ciampa}. We first prove the convergence of the velocity adapting the proof of \cite{bi} to this setting. Then, by using Lagrangian techniques, we show strong convergence of the vorticity providing a certain quantification of the convergence. Nonetheless, the rate of convergence is not fully quantitative since it depends (logarithmically) on the rate of convergence of the velocities.

Weak solutions of the Euler equations with vorticity in $L^p$ with $p\geq 3/2$ conserve the kinetic energy, but this is not known for $p < 3/2$ (cf. \cite{cheski}).  In Theorem \ref{thm:strongconvvort}, we also prove that limit solutions are energy conserving, for every $p \in (1, \infty]$, as previously done for the vanishing viscosity in \cite{cheski} and for the vortex blob in \cite{ciampa2}.

The second main result of this paper concerns the study of the rate of convergence in the case of solutions belonging to the Yudovich class. In particular, the next theorem shows that if we consider bounded initial vorticity, we can obtain a rate of convergence independent on $\|u^\alpha-u\|_{L^1_tL^1_x}$.

\begin{thm}\label{thm:rate} 
Let $\omega_0 \in L^\infty(\mathbb{T}^2)$ with $\int_{\mathbb{T}^2}\omega_0 = 0$. Let $q^\alpha_0$ uniformly bounded with respect to $\alpha$ in $L^\infty(\mathbb{T}^2)$ with $\int_{\mathbb{T}^2}q_0^\alpha = 0$ such that $q^\alpha_0 \rightarrow \omega_0$ in $L^p(\mathbb{T}^2)$, for every $p <\infty$.
Let $u^\alpha_0 :=(\identity- \alpha\lapl)^{-1} k \ast q^\alpha_0$ and $u_0 :=k \ast \omega_0$, then
\begin{equation} 
	\gamma_0^\alpha:=\|u^\alpha_0-u_0\|_{L^2} +\alpha \|\lapl u^\alpha_0\|_{L^2}\xrightarrow{\alpha \to 0} 0 \quad \text{and let} \quad\overline{\alpha}>0\quad \text{be s.t.}\quad \gamma_0^\alpha \leq \frac{1}{2},\quad \forall \alpha\leq \overline{\alpha}.
	\label{eq:closeenough} 
\end{equation}
 Let $(u^\alpha, q^\alpha)$ be the solution to the $\alpha$-Euler equations with initial datum $q^\alpha_0$ and let $(u, \omega)$ be its limit, which is the unique Yudovich solution to the Euler equations. 
 Then, there exist two constants $C_1$ and $C_2$ (depending on $M:=\|\omega_0\|_{L^\infty}$) such that, if $\alpha\in (0, \overline{\alpha}]$ and $T>0$ satisfy
\begin{equation} 
\alpha \leq \frac{\exp\{ 2(2- 2 \exp(C_2T))\}-\gamma_0^\alpha}{(C_1T)^2}, \label{eq:relalfaT}
\end{equation} 
it holds
\begin{displaymath} 
\|u(t)-u^\alpha(t)\|_{L^2}\leq \exp\{ 2- 2 \exp(-C_2t)\}(C_1\sqrt{\alpha}T+\gamma^\alpha_0)^{\exp(-C_2t)}+ C \sqrt{\alpha}: = K(\alpha,t), \quad \forall t \leq T.
\end{displaymath}
Moreover, there exists a value $\alpha_0 = \alpha_0(T,M, \omega_0)$ and a continuous function $\psi_{\omega_0,p,M}: \mathbb{R}^+\rightarrow\mathbb{R}^+$ vanishing at zero, such that for every $\alpha \leq \alpha_0$ there holds 
\begin{equation} 
\sup_{t \in [0,T]}\|q^\alpha(t) -\omega(t)\|_{L^p}\leq C M^{1-\frac{1}{p}}\max\left\{ \psi_{\omega_0,p,M}(K(\alpha,t)),K(\alpha,t)^{\frac{\exp{(-CT)}}{2p}} \right\},\label{eq:modulcontest} 
\end{equation} 
where $C$ depends on $M= \|\omega_0\|_{L^\infty}$.
\end{thm} 
One can check that thanks to \eqref{eq:closeenough}, it is always possible to find $\alpha$ small enough such that \eqref{eq:relalfaT} holds for a fixed time $T$ and vice versa a positive time $T$ such that \eqref{eq:relalfaT} holds for $\alpha\in (0, \overline{\alpha}]$.
Finally, assuming additional regularity on the initial datum, it is possible provide an explicit expression for the function $\psi_{\omega_0,p,M}$. 
\begin{corollary}\label{cor:besov} 
Under the same assumptions of \textsc{Theorem \ref{thm:rate}}, let $s \in (0,1)$, if $\omega_0$ belongs to the Besov space $B^s_{p, \infty}(\mathbb{T}^2)$, then the function $\psi_{\omega_0,p,M}(\cdot)$ is controlled, yielding 
\begin{equation} 
\sup_{t \in (0,T)}\|q^\alpha(t) -\omega(t)\|_{L^p}\leq C M^{1-\frac{1}{p}}\max\left\{ K(\alpha,t)^s,K(\alpha,t)^{\frac{\exp{(-CT)}}{2p}} \right\}, \label{eq:ratebesov} 
\end{equation} 
where $C$ depends on $M$.
\end{corollary} 
\bigskip 

Theorem \ref{thm:rate} is the analogous of the results obtained in \cites{const, ciampa,seiswied} for the vanishing viscosity limit and Corollary \ref{cor:besov} corresponds to \cite{const}, \textsc{Corollary 2}, for the vanishing viscosity limit. In particular, the interest in considering initial vorticity in $B^s_{p,\infty}$ relies on the fact that some classes of vortex patch are in those spaces. Indeed, in \cites{cowu2} it has been proved that if $\chi_\Omega$ is the characteristic function of $\Omega\subset \mathbb{R}^2$ whose boundary $\partial\Omega$ has box-counting dimension $\dim_{F}(\Omega)<2$, then 
\begin{equation*} 
\chi_{\Omega} \in B^{\frac{2-\dim_{F}(\Omega)}{p}}_{p, \infty}(\mathbb{T}^2), \quad \forall p \in [1, \infty).
\end{equation*} 

Finally, we mention that in the case of the $\alpha$-Euler equations, when considering a vortex patch of boundary $\mathcal{C}^{1,\gamma}$, in \cite{lt} the authors obtained a rate of convergence of order $\alpha^\frac{3}{4}$ according to our notation. The proof of \cite{lt} is built upon the observations that the $C^{1, \gamma}$ regularity of the boundary of a vortex patch is preserved in time and the vortex patches under this assumption belong to the space $L^\infty(0,T;{B}^{\frac{1}{2}}_{2,\infty}(\mathbb{T}^2))$.
\subsubsection*{Organization of the paper} In \textsc{Section \ref{sec:prelim}}, we introduce the notations used throughout the paper and we recall some standard results on the Euler equations and the known results on the well-posedness of the $\alpha$-Euler equations. \textsc{Section \ref{sec:wp}} is devoted to the proof of \textsc{Theorem \ref{thm:strongconvvort}}. In particular, the proof is split between \textsc{Proposition \ref{prop1}} where we prove strong convergence of the velocity and \textsc{Proposition \ref{prop2}} where we prove strong convergence of the vorticity. Finally, in \textsc{Section \ref{sec:chemin}} we prove \textsc{Theorem \ref{thm:rate}} and \textsc{Corollary \ref{cor:besov}} where we estimate the rate of convergence under the additional hypothesis of bounded initial datum. 
\paragraph*{Acknowledgments} This research has been partially supported by the {\em SNF Project 212573 FLUTURA – Fluids, Turbulence, Advection}. The first author acknowledges the hospitality of the University of Basel, where a large part of this work was carried out. The work of the third author is partially supported by INdAM-GNAMPA and by the project PRIN 2020 ”Nonlinear evolution PDEs, fluid dynamics and transport equations: theoretical foundations and applications”.

\section{Notations and preliminary results}
\label{sec:prelim}
\subsection{Notations}
Throughout this work, we always consider as a domain the two-dimensional flat torus $\mathbb{T}^2 = \mathbb{R}^2/\mathbb{Z}^2$ and we denote the Lebesgue measure on it by $\mathcal{L}^2$. The distance on the torus is defined as
\begin{displaymath}
\dist(x,y) = \min \{|x-y-k|: k \in \mathbb{Z}^2 \enskip|\enskip |k| <2\}
\end{displaymath}
and it is immediate to check that $\dist(x,y) \leq |x-y|$.\\
We use the standard definitions of functional spaces $L^p= L^p(\mathbb{T}^2)$, $H^s= H^s(\mathbb{T}^2)$ and $W^{s,p}=W^{s,p}(\mathbb{T}^2) $, in which we avoid writing explicitly the dependence on $\mathbb{T}^2$ for the norms. To simplify the notation, we use $L^p_tX_x$ for the Bochner spaces $L^p(0,T;X(\mathbb{T}^2))$. The space $\mathcal{C}^\infty_c([0,T) \times \mathbb{T}^2)$ denotes the space of smooth functions with compact support in time and periodic in space.\\
Finally, we introduce an ad-hoc norm for the solution of the $\alpha$-Euler equations, namely the $\alpha$-norm defined by
\begin{equation}
    \|\cdot\|_\alpha^2:= \|\cdot\|_{L^2}^2+\alpha\| \nabla(\cdot)\|_{L^2}^2.\label{eq:alphanorm}
\end{equation}
In this exposition, we use multiple times a generic constant $C$. This constant does not depend on $\alpha$ unless the dependence is specified and even if the constant $C$ appears more than one time in the same computation, its value may change from one line to the next.
\subsection{The $\alpha$-Euler equations}
For the sake of completeness, we derive the vorticity formulation of \eqref{eq:alphaeuler}. Let us define 
\begin{equation*}
    q^\alpha := \rotore v^\alpha \quad \text{and}\quad \omega^\alpha := \rotore u^\alpha. 
\end{equation*}
We recall that in dimension two the curl is a scalar and we employ the standard identities $\sum_j u^\alpha_j \nabla u^\alpha_j= \nabla |u^\alpha|^2/2$ and $\rotore(u^\alpha\cdot \nabla u^\alpha)= u^\alpha \cdot \nabla\omega^\alpha$. Taking the $\rotore$ of \eqref{eq:alphaeuler}, we get
\begin{equation}
    \partial_t q^\alpha+u^\alpha \cdot \nabla\omega^\alpha- \alpha \rotore (u^\alpha\cdot \nabla   (\lapl u^\alpha))- \alpha \rotore \left( \sum_j \lapl u^\alpha_j \nabla u^\alpha_j\right) = 0.\label{eq:contivort2}
\end{equation}
We compute the two curls and obtain
\begin{multline}
	\rotore (u^\alpha\cdot \nabla ( \lapl u^\alpha))= \partial_2 u_1^\alpha \partial_1 (\lapl u_1^\alpha) +u_1^\alpha \partial_2 \partial_1 (\lapl u_1^\alpha)+\partial_2 u_2^\alpha \partial _2 (\lapl u_1^\alpha) +u_2^\alpha \partial_2 ^2 (\lapl u_1^\alpha)\\
	-\partial_1 u_1^\alpha \partial_1 (\lapl u_2^\alpha)- u_1^\alpha \partial_1 ^2 (\lapl u_2^\alpha) -\partial_1 u_2^\alpha \partial_2 (\lapl u_2^\alpha) -u_2^\alpha \partial_1 \partial_2 (\lapl u_2^\alpha)\label{eq:contivort3}
\end{multline}
and
\begin{multline}
	\rotore\left(\sum_j \lapl u^\alpha_j \nabla u^\alpha_j\right)= \partial_2 (\lapl u_1^\alpha) \partial_1 u_1^\alpha+  \lapl u_1^\alpha \partial_2\partial_1 u_1^\alpha+\lapl u_2^\alpha\partial_2  \partial_1 u_2^\alpha+ \lapl u_2^\alpha \partial_2\partial_1 u_2^\alpha\\
	- \partial_1 (\lapl u_1^\alpha) \partial_2 u_1^\alpha-  \lapl u_1^\alpha \partial_1\partial_2 u_1^\alpha- \partial_1 (\lapl u_2^\alpha) \partial_2 u_2^\alpha- \lapl u_2^\alpha \partial_1\partial_2 u_2^\alpha. \label{eq:contivort4}
\end{multline}
Therefore, by \eqref{eq:contivort3}-\eqref{eq:contivort4}, simplifying the opposite terms, identity \eqref{eq:contivort2} becomes
\begin{equation*}
	\partial_t q^\alpha+u^\alpha \cdot \nabla\omega^\alpha- \alpha(u^\alpha \cdot \nabla)\lapl\omega^\alpha-\diver u^\alpha \rotore(\lapl u^\alpha) = 0.
\end{equation*}
Employing the incompressibility constraint, we obtain that $q^\alpha = \rotore v^\alpha$ satisfies
\begin{flalign}
\begin{cases}
\partial_t q^\alpha +u^\alpha\cdot \nabla q^\alpha =0,&\quad \text{on}  \quad(0,T) \times \mathbb{T}^2 \\
q(0, \cdot) = q_0^\alpha,&\quad \text{on} \quad\mathbb{T}^2,
\end{cases}\label{eq:vortalphaeuler}
\end{flalign}
where $v^\alpha$ is related to the vorticity $q^\alpha$ as
\begin{flalign}
\begin{cases}
\diver v^\alpha=0,&\quad \text{on}  \quad(0,T) \times \mathbb{T}^2 \\
\rotore v^\alpha= q^\alpha,&\quad \text{on}  \quad(0,T) \times \mathbb{T}^2 .
\end{cases}\label{eq:vortalphaeuleraux}
\end{flalign}
The system \eqref{eq:vortalphaeuleraux} yields the existence of a stream function $\psi^\alpha: [0,T) \times \mathbb{T}^2 \rightarrow \mathbb{R}$ such that $v^\alpha = \nabla^ \perp \psi^\alpha$ and
$ q^\alpha =-\lapl \psi^\alpha $ on $[0,T) \times \mathbb{T}^2$. The solution of this Poisson equation is given in terms of the Green function on the torus, under the condition $\int_{\mathbb{T}^2}q^\alpha dx = 0$, which is preserved in time by the equation at least formally. The Green function on the torus reads
\begin{displaymath}
G_{\mathbb{T}^2}(x,y) = \sum_{k \in \mathbb{Z}^2, k \neq 0}- \frac{\log|x-y-2 \pi k|}{2 \pi}
\end{displaymath}
and the corresponding Biot-Savart kernel, which can be used to represent the solution, reads
\begin{equation}
v^\alpha = \int_{\mathbb{T}^2} \nabla^\perp G_{\mathbb{T}^2}(x,y)q^\alpha(x) dx= k \ast q^\alpha. \label{eq:biotsavart}
\end{equation}
This relation implies that the Calderon-Zygmund estimates hold, see \cite{marchioro}, namely 
\begin{equation}
      \|\nabla v^\alpha \|_{L^p}\leq C_p\| q^\alpha\|_{L^p}, \quad \forall p \in (1,\infty),\label{eq:calderonzyg}
\end{equation}
where $C_p \sim p$ for $p > 2$.\\
We state the well-posedeness of the $\alpha$-Euler equations whose proof can be adapted from \cites{bi, bill} to the two-dimensional torus.
\begin{thm}[Well-posedeness of $\alpha$-Euler equations]\label{thm:wpalpha}
	Let $ q^\alpha_0 \in L^p(\mathbb{T}^2)$ with $p \in (1,\infty)$ and $\int_{\mathbb{T}^2}q^\alpha_0=0$. Then $u^\alpha_0 = (\identity-\alpha\lapl)^{-1}k \ast q^\alpha_0 \in W^{3,p}(\mathbb{T}^2)$ and there exists a unique solution $u^\alpha \in L^\infty(0,T;W^{3,p}(\mathbb{T}^2))$ of \eqref{eq:alphaeuler}. Moreover, the solution $u^\alpha$ conserves the $\alpha$-norm, namely
	\begin{equation}
		\|u^\alpha(t)\|_{\alpha}= \|u^\alpha_0\|_\alpha, \qquad \forall 0 \leq t \leq T.\label{eq:consalphanorm}
	\end{equation}
\end{thm}
The conservation of the $\alpha$-norm can be shown formally by considering \eqref{eq:alphaeuler} and testing it against $u^\alpha$, integrating over the torus, which yields
	\begin{multline*}
		\int_{\mathbb{T}^2}\partial_t (u^\alpha-\alpha\lapl u^\alpha)\cdot u^\alpha dx +\int_{\mathbb{T}^2}u^\alpha\cdot \nabla (u^\alpha - \alpha \lapl u^\alpha)\cdot u^\alpha dx \\+\int_{\mathbb{T}^2}\sum_j (u^\alpha-\alpha\lapl u^\alpha)_j \nabla u^\alpha_j \cdot u^\alpha dx  = -\int_{\mathbb{T}^2}\nabla p\cdot u^\alpha dx .
	\end{multline*}
	Exploiting $\diver u^\alpha=0$ and integrating by part, removing the trivially zero terms, we obtain
	\begin{equation*}
		\totdertime \frac{\|u^\alpha\|_\alpha^2}{2}= \int_{\mathbb{T}^2} \sum_{j,i=1}^2\partial_i(\lapl u^\alpha_j)u^\alpha_i u^\alpha_jdx -\int_{\mathbb{T}^2} \sum_{j,i=1}^2u^\alpha_j\partial_j (\lapl u^\alpha_i) u^\alpha_idx= 0
	\end{equation*}
	where the two terms simplified by swapping the indices in the sum.

\subsection{The two-dimensional Euler equations}
Let $T>0$ be arbitrary and finite, the two-dimensional Euler equations on the torus are
\begin{flalign}
	\begin{cases}
		\partial_t u +u\cdot \nabla u = -\nabla p, &\quad \text{on}  \quad(0,T) \times \mathbb{T}^2 \\
		\diver u = 0,&\quad \text{on}\quad (0,T) \times \mathbb{T}^2 \\
		u(0, \cdot) = u_0 ,&\quad \text{on} \quad\mathbb{T}^2,
	\end{cases}\label{eq:stdeuler}
\end{flalign} which in vorticity formulation reads
\begin{flalign}
	\begin{cases}
		\partial_t \omega  +u\cdot \nabla \omega =0,&\quad \text{on}  \quad(0,T) \times \mathbb{T}^2 \\
		\omega(0, \cdot) = \omega_0,&\quad \text{on} \quad\mathbb{T}^2,
	\end{cases}\label{eq:vorteuler}
\end{flalign}
where 
\begin{flalign}
	\begin{cases}
		\diver u=0,&\quad \text{on}  \quad(0,T) \times \mathbb{T}^2 \\
	u = k \ast \omega, &\quad \text{on}  \quad(0,T) \times \mathbb{T}^2.
	\end{cases}\label{eq:vorteuleraux}
\end{flalign}
We introduce the Lagrangian description of \eqref{eq:vorteuler}. Let $X:[0,T) \times [0,T) \times \mathbb{T}^2 \rightarrow\mathbb{T}^2$ be such that
\begin{flalign}
\begin{cases}
\dot{X}_{t,s}(x) = u(s,X_{t,s}(x)), \quad \forall s \in [0,T], \quad \forall x \in \mathbb{T}^2\\
X_{t,t}(x)=x, \quad \forall x \in \mathbb{T}^2,\label{eq:flowode}
\end{cases}
\end{flalign}
for any given $t \in [0,T]$. By the theory of characteristics, if $u$ is smooth, we know that the unique solution of the two-dimensional Euler equations with initial datum $\omega_0$ satisfies
\begin{equation}
    u(t,x) := (k \ast \omega)(t,x), \quad \text{and} \quad \omega(t,x) := \omega_0(X_{t,0}(x)).\label{eq:lagsol}
\end{equation}
In order to extend the definition to the non-smooth case, we need to introduce the following definition.
\begin{definition}[Regular Lagrangian flow]
A map $X \in L^\infty((0,T) \times(0,T) \times \mathbb{T}^2)$ is called a Lagrangian flow for the vector field $u \in L^1(0,T; L^1(\mathbb{T}^2))$ if
\begin{itemize}
    \item the map $s \mapsto X_{t,s}(x)$ is an absolutely continuous solution of \eqref{eq:flowode} for almost every $x \in \mathbb{T}^2$ and any $t \in [0,T)$;
    \item the map $x \mapsto X_{t,s}(x)$ is measure preserving with respect to the Lebesgue measure on the torus for any $s,t \in [0,T)$.
\end{itemize}\label{def:lagflow}
\end{definition}
The definition of Lagrangian solution in the non-smooth setting is the following.
\begin{definition}[Lagrangian solution to the Euler equations]
Let $p \in (1, \infty)$ and $\omega_0 \in L^p(\mathbb{T}^2)$. The couple $(u, \omega) \in L^\infty_t W^{1,p}_x \times L^\infty_t L^p_x $ is called Lagrangian solution to the two-dimensional Euler equations if there exists a regular Lagrangian flow in the sense of \textsc{Definition \ref{def:lagflow}} and the couple $(u,\omega)$ satisfies \eqref{eq:lagsol} for almost every $(t,x) \in (0,T) \times \mathbb{T}^2$.\label{def:lagsol}
\end{definition}
\section{Quantitative strong convergence of the vorticity}
\label{sec:wp}
In this section, we present the proof of \textsc{Theorem \ref{thm:strongconvvort}} which is split between \textsc{Proposition \ref{prop1}} and \textsc{Proposition \ref{prop2}}. We recall and adapt some lemmas introduced in \cite{bi}. In the first proposition, we prove the convergence in velocity and we show it implies conservation of energy. In the second proposition, we introduce the proof of convergence in vorticity through Lagrangian techniques.
\subsection{Preliminaries}
In this paragraph, we show some bounds which are needed to prove the convergences in \textsc{Theorem \ref{thm:strongconvvort}}.
We know that the solution $u^\alpha$ is regular enough so that the $L^p$ norms of $q^\alpha$ are preserved in time, thanks to \textsc{Theorem \ref{thm:wpalpha}}. Hence, using standard elliptic estimate and the Calderon-Zygmund inequality \eqref{eq:calderonzyg}, we get 
\begin{equation}
\|u^\alpha\|_{W^{1,p}} \leq C\|v^\alpha\|_{W^{1,p}} \leq C \| q^\alpha\|_{L^p}=C\| q^\alpha_0\|_{L^p}, \quad \forall p \in (1, \infty). \label{eq:w1pbound}
\end{equation}
Moreover, the elliptic equation $v^\alpha=u^\alpha -\alpha \lapl u^\alpha$ yields the additional regularity
\begin{equation}
      \|u^\alpha\|_{W^{3,p}} \leq  C\alpha^{-1}\|v^\alpha\|_{W^{1,p}} \leq C\alpha^{-1} \| q^\alpha\|_{L^p}=C\alpha^{-1}\| q^\alpha_0\|_{L^p}, \quad \forall p \in (1, \infty). \label{eq:w3pbound}
\end{equation}
This elliptic estimate can be adapted on the torus from \cite{gilbargtru}, \textsc{Theorem 8.10}. The inequality \eqref{eq:w3pbound} is not uniform with respect to $\alpha$; nevertheless, it is used in \textsc{Lemma \ref{lemma:iftimie}} to produce an improvement of \eqref{eq:w1pbound} to control the $L^2$ norm of the gradient and the Laplacian of $u^\alpha$, even for $p<2$. The proof is an adaptation of \textsc{Proposition 3.1} in \cite{bi}. This bound refines the proof for the convergence in velocities with respect to \cite{lt}. Moreover, the energy conservation for the limit solution is proven as a direct consequence of this estimate.
\begin{lemma}\label{lemma:iftimie}
Let $q^\alpha \in L^p(\mathbb{T}^2)$ and let $u^\alpha:= (\identity-\alpha\lapl)^{-1}k\ast q^\alpha$. Then, for every $t \in [0,T)$, it holds
\begin{align*}
   \|\nabla u^\alpha(t)\|_{L^2}& \leq C \alpha^{\frac{1}{2}-\frac{1}{p}}\| q^\alpha(t)\|_{L^p}, \qquad \forall p \in(1,2],\\
   \|\lapl u^\alpha(t)\|_{L^2}& \leq C \alpha^{-\frac{1}{p}}\| q^\alpha(t)\|_{L^p}, \qquad \forall p \in(1,2],\\
   \|\lapl u^\alpha(t)\|_{L^2}& \leq C \alpha^{-\frac{1}{2}}\| q^\alpha(t)\|_{L^p}, \qquad \forall p \geq 2.
\end{align*}
\end{lemma}
\begin{proof}
For $p \in (1,2]$, we interpolate \eqref{eq:w1pbound}-\eqref{eq:w3pbound} using Gagliardo-Nirenberg inequality. In particular to bound $\nabla u^\alpha$, we consider the following case
\begin{equation}
	\|\nabla u^\alpha\|_{L^2} \leq C \|\differ^3 u^\alpha\|_{L^p}^{\frac{1}{p}-\frac{1}{2}} \|\nabla u^\alpha\|_{L^p}^{\frac{3}{2}-\frac{1}{p}}, \qquad \forall p \in(1,2].
\end{equation}
Owing to \eqref{eq:w1pbound}-\eqref{eq:w3pbound}, we obtain
\begin{equation*}
    \|\nabla u^\alpha\|_{L^2} \leq C \alpha^{\frac{1}{2}-\frac{1}{p}} \|q^\alpha\|_{L^p},\qquad \forall p \in(1,2].
\end{equation*}
We proceed in analogous way to control $\|\lapl u^\alpha\|_{L^2}$ and we infer
\begin{equation*}
	\|\lapl u^\alpha\|_{L^2} \leq C\| \differ^3 u^\alpha \|_{L^p}^{\frac{1}{p}}\|\nabla u^\alpha \|_{L^p}^{1-\frac{1}{p}}\leq C \alpha^{-\frac{1}{p}}\|q^\alpha\|_{L^p},\qquad \forall p \in(1,2].
\end{equation*}
Let us define $v^\alpha = k \ast q^\alpha$, thus
\begin{displaymath}
	v^\alpha= u^\alpha-\alpha\lapl u^\alpha.
\end{displaymath}
We test this identity against $- \lapl u^\alpha$ to infer
\begin{equation}
	\|\nabla u^\alpha\|^2_{L^2}+\alpha \| \lapl u^\alpha\|^2_{L^2} =(\nabla v^\alpha,\nabla u^\alpha)_{L^2} \implies\frac{1}{2}\|\nabla u^\alpha\|^2_{L^2}+\alpha \| \lapl u^\alpha\|^2_{L^2} \leq \|\nabla v^\alpha\|^2_{L^2}, 
\end{equation}
by Young inequality. 
Lastly, employing \eqref{eq:w1pbound}, we deduce$$\sqrt{\alpha} \|\lapl u^\alpha\|_{L^2} \leq C \|q^\alpha\|_{L^p}, \qquad \forall p \geq 2,$$ which concludes the proof.
\end{proof}
\subsection{Strong convergence of the velocity and energy conservation}
At this point, we have all the tools to show the strong convergence of the velocity.
\begin{prop}\label{prop1}
	Let $T>0$ arbitrary and finite, let $q^\alpha_0$ and $\omega_0$ under the assumptions of \textsc{Theorem \ref{thm:strongconvvort}} and let the couple $(u^\alpha,q^\alpha)$ be the corresponding global solution to the $\alpha$-Euler equations. Then, there exists a couple $(u, \omega)$ such that, up to subsequences, it holds
	\begin{equation}
	u^\alpha \rightarrow u \quad \text{strongly in} \enskip \mathcal{C}([0,T];L^{2}(\mathbb{T}^2)),
\end{equation}	
\begin{equation}
	q^\alpha \rightharpoonup^* \omega \quad \text{weakly-$*$ in} \enskip L^\infty(0,T;L^p(\mathbb{T}^2))
\end{equation}
	and $u$ is a distributional solution of the Euler equations with initial datum $u_0 = k\ast \omega_0$. Finally, $u$ conserves the kinetic energy, namely
	\begin{equation*}
		\|u(t)\|_{L^2} = \|u_0\|_{L^2}, \quad \forall t \in (0,T).
	\end{equation*}
\end{prop}
\begin{proof}\textsc{\textbf{Step 1}: Weak convergences.}
We know that the sequence $q^\alpha$ is bounded uniformly in $L^\infty(0,T;L^p(\mathbb{T}^2))$ by \textsc{Theorem \ref{thm:wpalpha}}. By a standard compactness argument, we have that up to (non relabelled) subsequences
\begin{equation}
    q^\alpha \rightharpoonup^* \omega \quad \text{weakly-$*$ in} \enskip L^\infty(0,T;L^p(\mathbb{T}^2)).\label{eq:weakconvq}
\end{equation}
 By \eqref{eq:w1pbound}, the corresponding velocities $u^\alpha$ are bounded uniformly in $\alpha$ with respect to the $L^\infty_t W^{1,p}_x$-norm, hence, up to a (sub)subsequence, we have
\begin{equation}
   u^\alpha \rightharpoonup^* u \quad \text{weakly-$*$ in} \enskip L^\infty(0,T;W^{1,p}(\mathbb{T}^2)),\label{eq:weakconvu1}
\end{equation}
which implies
\begin{equation}
   u^\alpha \rightharpoonup^* u \quad \text{weakly-$*$ in} \enskip L^\infty(0,T;L^{2}(\mathbb{T}^2)).\label{eq:weakconvu}
\end{equation}
\textsc{\textbf{Step 2}: Strong convergence of the velocity.}
We rewrite \eqref{eq:alphaeuler} as
\begin{multline}
	\partial_t (u^\alpha -\alpha \lapl u^\alpha)=-\diver (u^\alpha \otimes u^\alpha)\\ 
	+\alpha \sum_{i,j} \partial_j \partial_i (u^\alpha_j \partial_i u ^\alpha)-\alpha \sum_{i,j} \partial_j (\partial_i u^\alpha_j \partial_i u^\alpha) +\alpha \sum_{i,j} \partial_i( \partial_i u_j^\alpha\nabla u_j^\alpha)-\nabla \pi^\alpha, \label{eq:euleralphalong}
\end{multline}
thanks to standard tensor identities and  $\pi^\alpha= p+|u^\alpha|^2/2$. \\
Now, we want to control the time derivative of the velocity and use Aubin-Lions Lemma, in order to show the strong convergence of $u^\alpha$. Let $\varphi \in \mathcal{C}^\infty(\mathbb{T}^2)$, then there exists $\varphi^\alpha := (1-\alpha \lapl)^{-1} \varphi$ and $\varphi^\alpha \in \mathcal{C}^\infty(\mathbb{T}^2)$.
Multiplying \eqref{eq:euleralphalong} by $\varphi^\alpha$ and integrating over $\mathbb{T}^2$, after some integrations by parts, we infer
\begin{multline}
	\int_{\mathbb{T}^2} \partial_t (u^\alpha -\alpha \lapl u^\alpha) \cdot  \varphi^\alpha dx= \int_{\mathbb{T}^2}  (u^\alpha \otimes u^\alpha):\nabla \varphi^\alpha dx+ \alpha \sum_{j,i} \int_{\mathbb{T}^2} u^\alpha_j \partial_i u^\alpha \cdot \partial_j \partial_i \varphi^\alpha dx\\+ \alpha \sum_{j,i} \int_{\mathbb{T}^2}\partial_i u^\alpha_j \partial_i u^\alpha \cdot \partial_j\varphi^\alpha dx- \alpha \sum_{j,i} \int_{\mathbb{T}^2}\partial_i u^\alpha_j \nabla u^\alpha_j \cdot \partial_i\varphi^\alpha dx-\int_{\mathbb{T}^2}\nabla \pi^\alpha\cdot \varphi^\alpha dx \\=T_1+T_2+T_3+T_4-\int_{\mathbb{T}^2}\nabla \pi^\alpha\cdot \varphi^\alpha dx.\label{eq:testedlong}
\end{multline}
On the left hand side of \eqref{eq:testedlong}, we get
\begin{equation}
	\int_{\mathbb{T}^2} \partial_t (u^\alpha -\alpha \lapl u^\alpha) \cdot  \varphi^\alpha dx= \int_{\mathbb{T}^2} \partial_t u^\alpha \cdot  (\idnt-\alpha\lapl)\varphi^\alpha dx= \int_{\mathbb{T}^2} \partial_t u^\alpha \cdot  \varphi dx.  \label{eq:lhspassage}
\end{equation}
Hence, we need to bound every term on the right hand side of \eqref{eq:testedlong} to control $\{\partial_t u^\alpha\}$.
The control on the first term in the right hand side of \eqref{eq:testedlong} is straightforward 
\begin{equation}
	|T_1 |:= \left|\int_{\mathbb{T}^2}(u^\alpha \otimes u^\alpha):\nabla \varphi^\alpha dx\right| \leq C \|u^\alpha\|^2_{L^2}\|\nabla \varphi^\alpha \|_{L^\infty} \leq C \|u^\alpha_0 \|^2_{\alpha} \|\varphi^\alpha\|_{H^3}.\label{eq:cita1}
\end{equation}
For the other terms, we employ the bound on the $\alpha$-norm given by \eqref{eq:consalphanorm}. In particular, for the second term of the right hand side of \eqref{eq:testedlong}, we obtain
\begin{equation}
	|T_2| :=\left| \alpha \sum_{j,i} \int_{\mathbb{T}^2} u^\alpha_j \partial_i u^\alpha \cdot \partial_j \partial_i \varphi^\alpha dx\right|
	\leq C \alpha \|u^\alpha \|_{L^2}\|\nabla u^\alpha\|_{L^2} \| \varphi^\alpha \|_{W^{2, \infty}} \leq C \alpha^{\frac{1}{2}}\|u^\alpha_0 \|^2_{\alpha}\| \varphi^\alpha \|_{H^4}. \label{eq:cita2}
\end{equation}
Finally, the third and fourth term are respectively bounded in the following way
\begin{align}
	|T_3|&:=\left| \alpha \sum_{j,i} \int_{\mathbb{T}^2}\partial_i u^\alpha_j \partial_i u^\alpha \cdot \partial_j\varphi^\alpha dx\right|
	\leq C \alpha \|\nabla u^\alpha \|_{L^2}^2\| \nabla \varphi^\alpha \|_{L^{ \infty}} \leq C \ \|u^\alpha_0 \|^2_{\alpha}\| \varphi^\alpha \|_{H^3}, \label{eq:cita3}\\
|T_4 |&:=\left| \alpha \sum_{j,i} \int_{\mathbb{T}^2}\partial_i u^\alpha_j \nabla u^\alpha_j \cdot \partial_i\varphi^\alpha dx\right|
	\leq C \alpha \|\nabla u^\alpha \|_{L^2}^2\| \nabla \varphi^\alpha \|_{L^{ \infty}} \leq C \ \|u^\alpha_0 \|^2_{\alpha}\| \varphi^\alpha \|_{H^3}. \label{eq:cita4}
\end{align}
Lastly, we need to estimate the pressure term. Let us consider \eqref{eq:euleralphalong} and take its divergence. Owing to the incompressibility constraint, we get
\begin{equation*}
	\partial_i \partial_j( u^\alpha_i u^\alpha_j )-\alpha\partial_l\left(  \partial_j \partial_i (u^\alpha_j \partial_i u ^\alpha_l)- \partial_j (\partial_i u^\alpha_j \partial_i u^\alpha_l) +  \partial_i( \partial_i u_j^\alpha\partial_l u_j^\alpha)\right) =- \lapl \pi^\alpha.
\end{equation*}
Here, we consider the pressure as the sum of two contributions $\pi^\alpha = \pi^\alpha_1+\pi^\alpha_2$ such that
\begin{equation*}
	\lapl \pi_1^\alpha =- \partial_i \partial_j( u^\alpha_i u^\alpha_j ), \quad \lapl \pi_2^\alpha =\alpha \partial_l\left( \partial_j \partial_i (u^\alpha_j \partial_i u ^\alpha_l)- \partial_j (\partial_i u^\alpha_j \partial_i u^\alpha_l) + \partial_i( \partial_i u_j^\alpha\partial_l u_j^\alpha)\right) . 
\end{equation*}
For the term $\pi^\alpha_1$, we notice that it can be bounded analogously to \eqref{eq:cita1}, which is
\begin{equation}
	\left|\int_{\mathbb{T}^2}\nabla \pi^\alpha_1 \cdot \varphi^\alpha dx\right| =   \left|\int_{\mathbb{T}^2} \partial_j (u_i^\alpha u_j^\alpha)\varphi^\alpha_j dx\right|\leq |T_1 |\leq C \|u^\alpha\|^2_{L^2}\|\varphi^\alpha\|_{H^3}.\label{eq:stimapre1}
\end{equation}\\
The control on the term $\pi^\alpha_2$ is exactly equivalent to the ones in \eqref{eq:cita2}-\eqref{eq:cita3}-\eqref{eq:cita4}, indeed
\begin{equation}
\left|	\int_{\mathbb{T}^2}\nabla \pi^\alpha_2 \cdot \varphi^\alpha dx\right| \leq |T_2|+|T_3|+|T_4| \leq C(\|u^\alpha\|^2_\alpha) \|\varphi^\alpha\|_{H^4}. \label{eq:stimapre2}
\end{equation}
Now, the estimates for the non-linear terms \eqref{eq:cita1}-\eqref{eq:cita2}-\eqref{eq:cita3}-\eqref{eq:cita4} and the ones for the pressure \eqref{eq:stimapre1}-\eqref{eq:stimapre2} complete the control of the right hand side of \eqref{eq:testedlong}. Indeed, employing the elliptic estimate $\| \varphi^\alpha \|_{H^4}\leq C\| \varphi \|_{H^4}$, we infer
\begin{equation}
	\langle \partial_t u^\alpha, \varphi \rangle \leq C \|\varphi^\alpha \|_{H^4 (\mathbb{T}^2)}\leq C \|\varphi \|_{H^4 (\mathbb{T}^2)}.\label{eq:H-sbound}
\end{equation}
Thus, $\partial_t u^\alpha$ is uniformly bounded with respect to $\alpha$ in $L^\infty_t H^{-4}_x$. 
The immersion of $L^2(\mathbb{T}^2)$ is continuous in $H^{-4}(\mathbb{T}^2)$ and by \eqref{eq:weakconvu1} the velocity converges weakly-$*$ in $L^\infty(0,T; W^{1,p}(\mathbb{T}^2))$ with $W^{1,p}(\mathbb{T}^2)$ compactly embedded in $L^2(\mathbb{T}^2)$. Hence, we use Aubin-Lions lemma to infer that up to a new (sub)subsequence
\begin{equation}
	u^\alpha \rightarrow u \quad \text{strongly in} \enskip \mathcal{C}([0,T];L^{2}(\mathbb{T}^2)).\label{eq:strongconvu}
\end{equation}
\textsc{\textbf{Step 3}: Equation for the velocity.} We want to show that the limit $u$ is a solution to the Euler equation with initial datum $u_0$. First, we recover strong convergence of $u^\alpha_0$ to $u_0$ in $L^2(\mathbb{T}^2)$.
We recall that the initial data are defined as 
\begin{equation*}
	v^\alpha_0:= k \ast q^\alpha_0, \quad u^\alpha_0:= (\identity-\alpha \lapl)^{-1} v^\alpha_0, \quad u_0:= k \ast \omega_0.
\end{equation*}  Owing to \eqref{eq:calderonzyg} and $q^\alpha_0 \to \omega_0$ in $L^p(\mathbb{T}^2)$ by hypothesis, we have that
\begin{equation}
	\|v^\alpha_0\|_{L^2}\to \|u_0\|_{L^2}.\label{eq:convinitene}
\end{equation}
Then, we consider \textsc{Lemma \ref{lemma:iftimie}} and if $p\geq 2$, we deduce
\begin{equation}
 \alpha\|\lapl u^\alpha_0\|_{L^2}=\sqrt{\alpha}(\sqrt{\alpha}\|\lapl u^\alpha_0\|_{L^2})\xrightarrow{\alpha \to 0} 0.\label{eq:newcit1}
\end{equation}
Whereas, if $p \in (1,2)$, we obtain
\begin{equation}
 \alpha\|\lapl u^\alpha_0\|_{L^2}=\alpha^{1-\frac{1}{p}}(\alpha^\frac{1}{p}\|\lapl u^\alpha_0\|_{L^2})\xrightarrow{\alpha \to 0} 0. \label{eq:newcit2}
\end{equation}
By \eqref{eq:newcit1}-\eqref{eq:newcit2}, we infer from \eqref{eq:convinitene} that
\begin{equation}
	\|u^\alpha_0-v^\alpha_0\|_{L^2} = \alpha\|\lapl u^\alpha_0\|_{L^2} \to 0 \implies 	\|u^\alpha_0\|_{L^2}\to \|u_0\|_{L^2}.\label{eq:laplpinit}
\end{equation}
Now, we are left to show that the limit $u$ is a distributional solution to the velocity formulation of the Euler equations. Thus, we need to pass to the limit into \eqref{eq:euleralphalong}. 
The term $ \partial_t (u^\alpha -\alpha \lapl u^\alpha) \rightarrow  \partial_t u$ in the sense of distribution thanks to \eqref{eq:weakconvu}. Indeed, for any $\varphi \in \mathcal{C}^\infty_c((0,T) \times \mathbb{T}^2)$, it holds
\begin{multline*}
	\langle \partial_t (u^\alpha -\alpha \lapl u^\alpha),\varphi\rangle =\int_{0}^T \int_{\mathbb{T}^2} u^\alpha\cdot (-\partial_t \varphi +\alpha \lapl \partial_t \varphi) dx ds\\ \to \int_{0}^T \int_{\mathbb{T}^2} u\cdot(-\partial_t \varphi) dx ds =	\langle \partial_t u,\varphi\rangle.	
\end{multline*} Let $p\in (1,2)$, considering the right hand side of \eqref{eq:euleralphalong}, thanks to \textsc{Lemma \ref{lemma:iftimie}}, we obtain
\begin{equation}
    \begin{aligned}
   \|\alpha u^\alpha_j \partial_i u ^\alpha\|_{L^1} &\leq C\alpha\|u^\alpha\|_{L^2}\|\nabla u^\alpha\|_{L^2}\leq C\alpha^{\frac{3}{2}-\frac{1}{p}}\|u^\alpha_0\|_{\alpha}\|q^\alpha_0\|_{L^p} \rightarrow 0,\\
   \| \alpha\partial_i u^\alpha_j \partial_i u^\alpha\|_{L^1} &\leq C\alpha\|\nabla u^\alpha\|^2_{L^2}\leq C\alpha^{2-\frac{2}{p}}\|q^\alpha_0\|^2_{L^p}\rightarrow 0,\\
   \| \alpha \partial_i u_j^\alpha\nabla u_j^\alpha\|_{L^1} & \leq C\alpha\|\nabla u^\alpha\|^2_{L^2}\leq C\alpha^{2-\frac{2}{p}}\|q^\alpha_0\|^2_{L^p}\rightarrow 0.
\end{aligned}\label{eq:citasottoqua}
\end{equation}
Analogously for  $p \in[2, \infty)$, we use \eqref{eq:w1pbound} to deduce
\begin{equation}
	\begin{aligned}
		\|\alpha u^\alpha_j \partial_i u ^\alpha\|_{L^1} &\leq C\alpha\|u^\alpha\|_{L^2}\|\nabla u^\alpha\|_{L^2}\leq C\alpha\|u^\alpha_0\|_{\alpha}\|q^\alpha_0\|_{L^p} \rightarrow 0,\\
		\| \alpha\partial_i u^\alpha_j \partial_i u^\alpha\|_{L^1} &\leq C\alpha\|\nabla u^\alpha\|^2_{L^2}\leq C\alpha\|q^\alpha_0\|^2_{L^p}\rightarrow 0,\\
		\| \alpha \partial_i u_j^\alpha\nabla u_j^\alpha\|_{L^1} & \leq C\alpha\|\nabla u^\alpha\|^2_{L^2}\leq C\alpha\|q^\alpha_0\|^2_{L^p}\rightarrow 0.
	\end{aligned}\label{eq:citasottoqua2}
\end{equation}
After integration by parts, for any $p \in (1,\infty)$ inequalities \eqref{eq:citasottoqua2}-\eqref{eq:citasottoqua} imply the following convergences in the sense of distribution
\begin{displaymath}
    \alpha \sum_{i,j} \partial_j \partial_i (u^\alpha_j \partial_i u ^\alpha) \rightarrow 0,\quad
    \alpha \sum_{i,j} \partial_j (\partial_i u^\alpha_j \partial_i u^\alpha)  \rightarrow 0\quad\text{and}\quad
    \alpha \sum_{i,j} \partial_i( \partial_i u_j^\alpha\nabla u_j^\alpha) \rightarrow 0.
\end{displaymath}
Hence, in the right hand side of \eqref{eq:euleralphalong}, we are left to pass to the limit for  $\diver(u^\alpha \otimes u^\alpha)$. However, this is implied by strong convergence of the velocity in \eqref{eq:strongconvu}.\\
\textsc{\textbf{Step 4}: Energy conservation.}
By \textsc{Theorem \ref{thm:wpalpha}}, we know that the $\alpha$-norm of the solution is conserved, namely
\begin{equation}
	\|u^\alpha(t)\|^2_{L^2}+\alpha\|\nabla u^\alpha(t)\|^2_{L^2}=	\|u^\alpha_0\|^2_{L^2}+\alpha\|\nabla u^\alpha_0\|^2_{L^2}\label{eq:recallalfanorm}
\end{equation}
and we want to pass to the limit as $\alpha \to 0$.
Considering
\textsc{Lemma \ref{lemma:iftimie}}, we obtain
\begin{equation}
	\alpha \|\nabla u^\alpha\|_{L^2}^2= \alpha^{2-\frac{2}{p}}  \left(\alpha^{\frac{1}{p}-\frac{1}{2}}\|\nabla u^\alpha\|_{L^2}\right)^2 \leq \alpha^{2-\frac{2}{p}}C^2 \rightarrow 0, \quad \text{as} \enskip \alpha\rightarrow 0, \quad \forall p \in (1,2).\label{eq:energyp<2}
\end{equation}
Moreover, by \eqref{eq:w1pbound} we get
\begin{equation}
	\alpha \|\nabla u^\alpha\|_{L^2}^2 \leq \alpha \|q^\alpha\|_{L^p}^2 \rightarrow 0,  \quad \text{as} \enskip \alpha\rightarrow 0, \quad \forall p \in [2, \infty].\label{eq:energyp>2}
\end{equation}
We proceed in analogous way to control $\nabla u^\alpha_0$ and we infer
\begin{equation}
	\alpha \|\nabla u^\alpha_0\|_{L^2}^2\rightarrow 0, \quad \text{as} \enskip \alpha \to 0, \quad \forall p \in (1,\infty].\label{eq:energypinit}
\end{equation}
Employing \eqref{eq:energyp<2}-\eqref{eq:energyp>2} and \eqref{eq:energypinit}, we pass to the limit into \eqref{eq:recallalfanorm} to obtain the thesis
\begin{equation}
	\|u(t)\|^2_{L^2}=\|u_0\|^2_{L^2}, \quad \forall p \in (1, \infty],
\end{equation}
where we have used the strong convergence of the velocity in $L^2(\mathbb{T}^2)$ expressed by \eqref{eq:strongconvu}-\eqref{eq:laplpinit}.
\end{proof}
\subsection{Strong convergence of the vorticities}
We proceed to prove that the limit solution of the $\alpha$-Euler equations is a Lagrangian solution of the Euler equations. This proof is analogous to the one in \cite{crippaspirito} for the vanishing viscosity scheme.
Let us introduce the transport equation as
\begin{flalign}\label{eq:transportgene}
	\begin{cases}
		\partial_t \rho +b\cdot \nabla \rho=0,&\quad \text{on}  \quad(0,T) \times \mathbb{T}^2 \\
		\rho(0, \cdot) = \rho_0,&\quad \text{on} \quad\mathbb{T}^2,
	\end{cases}
\end{flalign}
with $\diver b= 0$. We define in the following way a renormalized solution.
\begin{definition}\label{def:reno}
 A measurable function $\rho$ is a renormalized solution of \eqref{eq:transportgene}, if it solves in the sense of distribution
\begin{flalign*}
	\begin{cases}
		\partial_t \beta(\rho) +b\cdot \nabla \beta(\rho)=0,&\quad \text{on}  \quad(0,T) \times \mathbb{T}^2 \\
		\beta(\rho)(0, \cdot) = \beta(\rho_0),&\quad \text{on} \quad\mathbb{T}^2,
	\end{cases}
\end{flalign*}
for any $\beta \in \mathcal{C}^1(\mathbb{R}) \cap L^\infty(\mathbb{R})$.
\end{definition}
We first recall the following lemma, given by  \textsc{Theorem II.6}, \cite{dipernalions}.
\begin{lemma}\label{lemma:duality}
	Let $b$ be a vector field such that
	\begin{equation*}
		b(t,x) \in L^1(0,T;W^{1,p}(\mathbb{T}^2)), \quad \diver b= 0
	\end{equation*}
	and $\rho \in L^\infty(0,T; L^p (\mathbb{T}^2))$ be a renormalized solution of the transport equation according to \textsc{Definition \ref{def:reno}}. Let $\xi \in L^\infty(0,T; L^q(\mathbb{T}^2))$, where $q= \frac{p-1}{p}$, be a renormalized solution of the following backward transport problem
	\begin{flalign*}
		\begin{cases}
			- \partial_t \xi - \diver(b \xi) = \chi, & \quad \text{in} \quad [0,T) \times \mathbb{T}^2\\
			\xi(T,\cdot) = \xi_T, & \quad \text{in} \quad \mathbb{T}^2,
		\end{cases}
	\end{flalign*}
	where $\chi \in L^1(0,T;L^q(\mathbb{T}^2))$ and $\xi_T \in L^q(\mathbb{T}^2)$. Then, it holds
	\begin{equation}
		\int_0^T\int_{\mathbb{T}^2} \chi \rho dx ds =  \int_{\mathbb{T}^2}\xi(x,0)\rho_0(x) dx- \int_{\mathbb{T}^2}\xi_T(x)\rho(x,T) dx.\label{eq:flowestimate}
	\end{equation}
\end{lemma}
With the introduction of this setting, we are able to prove that the limit $\omega$ is a Lagrangian solution of the Euler equations.
\begin{prop}
	Let $(u,\omega)$ be the limit of $(u^\alpha,q^\alpha)$ according to \textsc{Proposition \ref{prop1}}. Then, $(u, \omega)$ is a Lagrangian solution of the Euler equations according to \textsc{Definition \ref{def:lagsol}}.
\end{prop}
\begin{proof}\textsc{\textbf{Step 1}: Consistent limit.}
We begin the proof by showing that $ u= k \ast \omega$. The starting point is the equation
	\begin{displaymath}
		u^\alpha-\alpha \lapl u^\alpha = k \ast q^\alpha, \quad \text{a.a.} \enskip t \in (0,T).
	\end{displaymath}
	Given a scalar test function $\varphi \in \mathcal{C}^\infty_c((0,T)\times\mathbb{T}^2)$, we have that
	\begin{displaymath}
		\int_0^T \int_{\mathbb{T}^2} (u^\alpha-\alpha \lapl u^\alpha)\varphi dx ds = \int_0^T \int_{\mathbb{T}^2} ( k \ast q^\alpha)\varphi dx ds.
	\end{displaymath}
	Considering the left hand side, we obtain 
	\begin{multline*}
		\lim_{\alpha \rightarrow 0}\int_0^T \int_{\mathbb{T}^2} (u^\alpha-\alpha \lapl u^\alpha)\varphi dx ds=\\ = \lim_{\alpha \rightarrow 0} \int_0^T \int_{\mathbb{T}^2} u^\alpha\varphi dx ds - \lim_{\alpha \rightarrow 0}\alpha  \int_0^T \int_{\mathbb{T}^2} u^\alpha\lapl\varphi dx ds =\int_0^T \int_{\mathbb{T}^2}u\varphi dx ds,
	\end{multline*}
	where we exploited the smoothness of $\varphi$ and \eqref{eq:strongconvu}.
	Instead, for the right hand side we have
	\begin{multline*}
		\int_0^T \int_{\mathbb{T}^2} ( k \ast q^\alpha)\varphi dx ds= \int_0^T \int_{\mathbb{T}^2} \left(\int_{\mathbb{T}^2}  k_{\mathbb{T}^2}(x-y)  q^\alpha(x,s)\right) \varphi(x,s) dx ds =\\= \int_0^T \int_{\mathbb{T}^2} q^\alpha(x,s)\left(\int_{\mathbb{T}^2}k_{\mathbb{T}^2}(x-y)\varphi(x,s)\right) dx ds=
		\int_0^T \int_{\mathbb{T}^2} q^\alpha(k \ast \varphi) dx ds.
	\end{multline*}
	Here, being $\varphi \in \mathcal{C}^\infty_c((0,T)\times\mathbb{T}^2)$, by Young inequality we have
	\begin{equation}
		\|k \ast \varphi\|_{L^q}\leq \|k\|_{L^1}\|\varphi\|_{L^q} \leq C.
	\end{equation}
	Employing \eqref{eq:weakconvq}, taking the limit as $\alpha \rightarrow 0^+$ and switching back the convolution, we infer
	\begin{displaymath}
		\int_0^T \int_{\mathbb{T}^2}u\varphi dx ds = \int_0^T \int_{\mathbb{T}^2} \omega(k\ast \varphi) dx ds = \int_0^T \int_{\mathbb{T}^2} ( k \ast \omega)\varphi dx ds,
	\end{displaymath}
	for any $\varphi \in \mathcal{C}^\infty_c((0,T)\times\mathbb{T}^2)$. This implies $u =  k \ast \omega$ for almost every $(t,x) \in(0,T)\times \mathbb{T}^2$ as wanted.\\
\textsc{\textbf{Step 2}: Lagrangian solution.}
We show that $(u, \omega)$ is a Lagrangian solution. We only need to prove it for $p \in(1,2)$, because for $p \geq 2$ it follows directly from uniqueness of the solution of the transport equation (see \cite{dipernalions}, \textsc{Theorem II.3}).
Let us consider \eqref{eq:vortalphaeuler}, let $\chi \in \mathcal{C}^\infty((0,T) \times\mathbb{T}^2)$ and let us define the backward transport problem for $u^\alpha$ of the form
\begin{flalign} 
\begin{cases}
- \partial_t \xi^\alpha - \diver(u^\alpha \xi^\alpha) = \chi ,& \quad \text{in} \quad (0,T) \times \mathbb{T}^2\\
\xi^\alpha(T,\cdot) = 0 ,& \quad \text{in} \quad \mathbb{T}^2
\end{cases}\label{eq:backalpha}
\end{flalign}
and the limit backward problem
\begin{flalign}
\begin{cases}
- \partial_t \xi - \diver(u \xi) = \chi ,& \quad \text{in} \quad (0,T) \times \mathbb{T}^2\\
\xi(T,\cdot) = 0, & \quad \text{in} \quad \mathbb{T}^2.
\end{cases}\label{eq:backxi}
\end{flalign}
Thanks to the stability theorem in DiPerna-Lions, \cite{dipernalions} \textsc{Theorem II.4}, it holds that $\xi^\alpha \rightarrow \xi$ in $\mathcal{C}([0,T];L^q(\mathbb{T}^2))$, for every $q \in [1, \infty]$. Since $u^\alpha$ is smooth, $q^\alpha$ satisfies 
\begin{equation}
    \int_{\mathbb{T}^2}\int_0^T \chi q^\alpha dx ds =  \int_{\mathbb{T}^2}\xi^\alpha(x,0)q_0^\alpha(x) dx,\label{eq:passtothelimitafter}
\end{equation}
where $\xi_\alpha$ solves \eqref{eq:backalpha}.
We recall that $q^\alpha \rightharpoonup^* \omega $ in $L^\infty(0,T;L^p(\mathbb{T}^2))$ and $q^\alpha_0 \rightarrow \omega_0$ in $L^p(\mathbb{T}^2)$, therefore passing to the limit in \eqref{eq:passtothelimitafter}, we obtain
\begin{equation}
    \int_{\mathbb{T}^2}\int_0^T \chi\omega dx ds =  \int_{\mathbb{T}^2}\xi(x,0)\omega_0(x) dx,\label{eq:citalag1}
\end{equation}
where $\xi$ is the unique solution in $\mathcal{C}([0,T];L^q(\mathbb{T}^2))$ of \eqref{eq:backxi}. Using \textsc{Lemma \ref{lemma:duality}} on the limit backward problem, we infer
\begin{equation}
    \int_{\mathbb{T}^2}\int_0^T \chi \omega_L dx ds =  \int_{\mathbb{T}^2}\xi(x,0)\omega_0(x) dx,\label{eq:citalag2}
\end{equation}
where $\omega_L$ is the unique renormalized solution, thus Lagrangian of the transport equation \eqref{eq:transportgene} with velocity field $u$ and initial datum $u_0$ (cf. \cite{dipernalions} \textsc{Theorem II.3}). Subtracting \eqref{eq:citalag1} and \eqref{eq:citalag2} we get
\begin{equation*}
    \int_{\mathbb{T}^2}\int_0^T \chi (\omega_L-\omega) dx ds= 0, \quad \forall \chi \in \mathcal{C}^\infty((0,T) \times\mathbb{T}^2),
\end{equation*}
which implies that $\omega = \omega_{L}$.
\end{proof}
 We notice that for the $\alpha$-Euler equations, we can introduce the flow map $X^\alpha$ using the classical theory of characteristic (cf. \eqref{eq:flowode}-\eqref{eq:lagsol}), since the velocity field $u^\alpha$ Lipschitz due to the embedding $W^{3,p} \hookrightarrow W^{1, \infty}$.
Knowing that the limit solution is Lagrangian for any $p \in (1,\infty)$, we want to study the convergence of the flows as it has been done in \cite{odeest}.
\begin{lemma}\label{lemma:flowsesti}
	Let $T>0$ arbitrary and finite and let $(u^\alpha_0,q^\alpha_0)$ be under the assumptions of \textsc{Theorem \ref{thm:strongconvvort}} and let $(u^\alpha, q^\alpha)$ be the sequence of unique solutions of $\alpha$-Euler equations according to \textsc{Theorem \ref{thm:wpalpha}}. Let $(u, \omega)$ be the limit of $(u^\alpha,q^\alpha)$ obtained in \textsc{Proposition \ref{prop1}} and let $X^\alpha_{t,s}$ and $X_{t,s}$ be the corresponding Lagrangian flows according to \eqref{eq:flowode} and \textsc{Definition \ref{def:lagflow}}.
Then, it holds
\begin{equation}
      \int_{\mathbb{T}^2}\dist(X^\alpha_{t,s}(x) ,X_{t,s}(x)) dx \leq
      \frac{C}{\left| \log\left(\|u^\alpha-u\|_{L^1_tL^1_x}\right)\right|} ,\label{eq:flowestigiusta}
\end{equation}
where the constant $C$ depends on $\|\nabla u\|_{L^1_tL^p_x}$ and on $T$.
\end{lemma}
\begin{proof}
Let $\delta= \|u^\alpha-u\|_{L^1_tL^1_x}$. Let us define the quantity
\begin{equation}
    g_\delta(s) := \int_{\mathbb{T}^2}\log\left(\frac{|X^\alpha_{t,s}(x) -X_{t,s}(x)|}{\delta}+1\right) dx. \label{eq:gdelta}
\end{equation}
Let us consider $ x \mapsto \log\left( 1+ \frac{x}{\delta}\right)$, increasing in $[0, \infty)$. We use Chebyshev inequality to infer
\begin{equation}
    \mathcal{L}^2 \left(\left\{ x\in \mathbb{T}^2 |  \dist(X^\alpha_{t,s}(x) ,X_{t,s}(x)) > \varepsilon\right\} \right)\leq \frac{1}{\log \left( \frac{\varepsilon}{\delta}+1\right)}\int_{\mathbb{T}^2}\log\left(\frac{|X^\alpha_{t,s}(x) -X_{t,s}(x)|}{\delta}+1\right) dx, \label{eq:cheby}
\end{equation}
for every $\varepsilon >0$. We have used that on the torus it holds $\dist(X^\alpha_{t,s}(x) ,X_{t,s}(x)) \leq |X^\alpha_{t,s}(x) -X_{t,s}(x)|$. We split the integral over the torus in two complementary sets as 
\begin{multline*}
    \int_{\mathbb{T}^2}\dist(X^\alpha_{t,s}(x),X_{t,s}(x)) dx \leq \int_{ \left\{ x\in \mathbb{T}^2 |  \dist(X^\alpha_{t,s}(x),X_{t,s}(x)) \leq \varepsilon\right\} }\dist(X^\alpha_{t,s}(x),X_{t,s}(x)) dx \\
    +\int_{ \left\{ x\in \mathbb{T}^2 |  \dist(X^\alpha_{t,s}(x),X_{t,s}(x)) >\varepsilon\right\} }\dist(X^\alpha_{t,s}(x),X_{t,s}(x)) dx.
\end{multline*}
We use \eqref{eq:cheby} to infer
\begin{equation}
    \int_{\mathbb{T}^2}\dist(X^\alpha_{t,s}(x),X_{t,s}(x)) dx \leq \varepsilon
    + \frac{1}{\log \left( \frac{\varepsilon}{\delta}+1\right)}\int_{\mathbb{T}^2}\log\left(\frac{|X^\alpha_{t,s}(x) -X_{t,s}(x)|}{\delta}+1\right) dx. \label{eq:subnu}
\end{equation}
The inequality \eqref{eq:subnu} holds true for any $\varepsilon >0$, thus we can choose $\varepsilon = \sqrt{\delta}$.
We recall that  $\delta$ goes to zero as $\alpha\rightarrow 0$, by definition. Therefore, we can take a value $\alpha$ for which $\delta <1$, which yields
\begin{equation*}
    \frac{1}{\log \left( \frac{1}{\sqrt{\delta}}+1\right)} \leq  \frac{1}{\left| \log \sqrt{\delta}\right|} =\frac{2}{\left| \log\delta\right|}.
\end{equation*}
Substituting into \eqref{eq:subnu} with \eqref{eq:gdelta}, we infer
\begin{equation*}
    \int_{\mathbb{T}^2}\dist(X^\alpha_{t,s}(x),X_{t,s}(x)) dx \leq \sqrt{\delta}
    + \frac{2}{\left| \log\delta\right|}g_\delta(s).
\end{equation*}
By \textsc{Definition \ref{def:lagflow}} and \eqref{eq:gdelta}, we know $g_\delta(t) = 0$ and
\begin{multline}
    g_\delta(s) =  \int_t^s  g_\delta'(\tau) d \tau \leq  \int_t^s \int_{\mathbb{T}^2} \frac{|\dot{X}^\alpha_{t, \tau}(x)-\dot{X}_{t, \tau}(x)|}{|X^\alpha_{t, \tau}(x)-X_{t, \tau}(x)|+\delta}dx d \tau \\
    \leq  \int_t^s \int_{\mathbb{T}^2} \frac{|u^\alpha(\tau,X^\alpha_{t,\tau})-u(\tau,X^\alpha_{t,\tau})|}{|X^\alpha_{t, \tau}(x)-X_{t, \tau}(x)|+\delta}dx d \tau + \int_t^s \int_{\mathbb{T}^2} \frac{|u(\tau,X^\alpha_{t,\tau})-u(\tau,X_{t,\tau})|}{|X^\alpha_{t, \tau}(x)-X_{t, \tau}(x)|+\delta}dx d \tau,\label{eq:citaproofflow}
\end{multline}
where we have summed $\pm u^\alpha(\tau,X^\alpha_{t,\tau})$ in the numerator and we have used the triangular inequality.
The first term in  the right hand side of \eqref{eq:citaproofflow} is controlled as
\begin{multline}
    \int_t^s \int_{\mathbb{T}^2} \frac{|u^\alpha(\tau,X^\alpha_{t,\tau})-u(\tau,X^\alpha_{t,\tau})|}{|X^\alpha_{t, \tau}(x)-X_{t, \tau}(x)|+\delta}dx d \tau  \\\leq\int_t^s \int_{\mathbb{T}^2}\frac{|u^\alpha(\tau,X^\alpha_{t,\tau})-u(\tau,X^\alpha_{t,\tau})|}{\delta}dx d \tau \leq \frac{\|u-u^\alpha\|_{L^1_tL^1_x}}{\delta}\leq C,\label{eq:citaflow1}
\end{multline}
thanks to \eqref{eq:citaproofflow}.
We are left to bound the second term in the right hand side of \eqref{eq:citaproofflow}.
Let $\mathcal{M}$ be the maximal function operator defined on $L^1$ functions as
\begin{equation*}
	\mathcal{M} f(x):= \sup_{r>0}\frac{1}{\mathcal{L}^2(B_r)} \int_{B_r(x)}|f(\xi)|d\xi, \quad \forall x \in \mathbb{T}^2.
\end{equation*}
We recall that for any $f \in L^p$ with $p \in (1,\infty]$, it holds
\begin{equation*}
	\|\mathcal{M}f\|_{L^p}\leq C\|f\|_{L^p}.
\end{equation*}
Moreover, for any $f \in W^{1,1}$, there exists a set $\mathcal{N}\subset \mathbb{T}^2$ such that $\mathcal{L}^2(\mathcal{N}) = 0$ and
\begin{equation}
	|f(x) -f(y) |  \leq C \dist(x,y) (\mathcal{M}\differ f(x) +\mathcal{M}\differ f(y)) , \quad  \forall  x,y \in \mathbb{T} \setminus \mathcal{N}.\label{eq:maxine}
\end{equation}
This allows us to deduce
\begin{multline}
    \int_t^s \int_{\mathbb{T}^2} \frac{|u(\tau,X^\alpha_{t,\tau})-u(\tau,X_{t,\tau})|}{|X^\alpha_{t, \tau}(x)-X_{t, \tau}(x)|+\delta}dx d \tau\leq C  \int_t^s \int_{\mathbb{T}^2}|\mathcal{M} \nabla u(\tau,X^\alpha_{t,\tau})|dx d\tau+\\ + C  \int_t^s \int_{\mathbb{T}^2}| \mathcal{M} \nabla u(\tau,X_{t,\tau})| dx d \tau 
 \leq C \|\nabla u \|_{L^1_tL^1_x}\leq C \|\nabla u \|_{L^1_tL^p_x}, \label{eq:citaflow2}
\end{multline}
thanks to the measure preserving property of the flows $X^\alpha_{t,s}$ and $X_{t,s}$. Substituting \eqref{eq:citaflow1} and \eqref{eq:citaflow2} into \eqref{eq:citaproofflow}, we obtain
\begin{equation*}
      \int_{\mathbb{T}^2}\dist(X^\alpha_{t,s}(x),X_{t,s}(x)) dx \leq \sqrt{\delta}
    + \frac{2}{\left| \log\delta\right|}C \leq \frac{C}{\left| \log\delta\right|} ,
\end{equation*}
exploiting that $\sqrt{\delta} \leq \frac{1}{\log(\delta)}$, for $\delta <1$.
Recalling that $\delta = \|u^\alpha-u\|_{L^1_t L^1_x}$, we have the thesis.
\end{proof}
\textsc{Lemma \ref{lemma:flowsesti}} allows to prove a result on the strong convergence of the vorticities.
\begin{prop}\label{prop2}
	Let $T>0$ arbitrary and finite and let $(u^\alpha_0,q^\alpha_0)$ be under the assumptions of \textsc{Theorem \ref{thm:strongconvvort}}. Let the couple $(u^\alpha,q^\alpha)$ be the corresponding global solution to the $\alpha$-Euler equations. Let $(u, \omega)$ be the limit of $(u^\alpha,q^\alpha)$ obtained in \textsc{Proposition \ref{prop1}}, then for any $\delta >0$, there exists $K(\delta, \omega_0)$ such that for $\alpha$ small enough it holds
\begin{equation}
    \sup_{t \in [0,T)}\|q^\alpha(t) -\omega(t) \|_{L^p} \leq \delta +\frac{K(\delta, \omega_0)}{\left| \log\left(\|u^\alpha-u\|_{L^1_tL^1_x}\right)\right|}+\|q^\alpha_0-\omega_0\|_{L^p},\label{eq:quantestcitar}
\end{equation}
which implies
\begin{equation}
	q^\alpha \rightarrow \omega \quad \text{strongly in} \enskip \mathcal{C}([0,T];L^p (\mathbb{T}^2)).\label{eq:strongconvvort}
\end{equation}
\end{prop}
\begin{proof}
We recall that Lipschitz function are dense in $L^p(\mathbb{T}^2)$, for every $p < \infty$. Thus, we take a sequence of Lipschitz function $\{\omega_0^j\}_{j \in \mathbb{N}}$, such that $\omega_0^j \rightarrow \omega$ as $j \rightarrow \infty$ in $L^p(\mathbb{T}^2)$. We notice
\begin{multline*}
    \|q^\alpha(t)-\omega(t)\|_{L^p} = \left(\int_{\mathbb{T}^2}|q_0^\alpha(X^\alpha_{t,0})-\omega_0(X_{t,0})|^pdx \right)^{\frac{1}{p}}\leq \\
    \left(\int_{\mathbb{T}^2}|\omega_0^j(X^\alpha_{t,0})-\omega_0(X^\alpha_{t,0})|^pdx \right)^{\frac{1}{p}}+\left(\int_{\mathbb{T}^2}|\omega_0^j(X_{t,0})-\omega_0(X_{t,0})|^pdx \right)^{\frac{1}{p}}+\\\left(\int_{\mathbb{T}^2}|\omega_0^j(X^\alpha_{t,0})-\omega_0^j(X_{t,0})|^pdx \right)^{\frac{1}{p}}+\left(\int_{\mathbb{T}^2}|q_0^\alpha(X^\alpha_{t,0})-\omega_0(X^\alpha_{t,0})|^pdx \right)^{\frac{1}{p}}.
\end{multline*}
Knowing that $\omega_0^j$ are Lipschitz and using the estimate on the flows \eqref{eq:flowestigiusta}, we deduce
\begin{equation}
    \|q^\alpha(t) -\omega(t) \|_{L^p} \leq 2\|\omega_0^j(t) -\omega_0(t)\|_{L^p}+\frac{K(\delta, \omega_0)}{ \left| \log\left(\|u^\alpha-u\|_{L^1_tL^1_x}\right)\right|}+\|q^\alpha_0(t) -\omega_0(t)\|_{L^p}. \label{eq:lip}
\end{equation}
If we send first $\alpha\rightarrow 0$ and then $j \rightarrow \infty$, by employing the strong convergence of $u$ in $L^\infty_t L^2_x$, thus in $L^1_t L^1_x$, we get that $q^\alpha \rightarrow \omega$ strongly in $\mathcal{C}_t L^p_x$. Given $\delta >0$, we can take the sequence $\{\omega_0^j\}_{j \in \mathbb{N}}$ such that $\delta \geq 2\|\omega_0^j(t) -\omega_0(t)\|_{L^p}$, for any $j \in \mathbb{N}$, and infer \eqref{eq:quantestcitar} from \eqref{eq:lip}.\\
The convergence \eqref{eq:strongconvvort} follows from $\| u^\alpha-u\|_{L^1_tL^1_x} \rightarrow 0$ as $\alpha \to 0$.
\end{proof}
This result completes the proof of \textsc{Theorem \ref{thm:strongconvvort}}.

\section{Rate of convergence for bounded vorticity}
\label{sec:chemin}
In this section, we analyse the rate of convergence for the initial vorticity in $L^\infty(\mathbb{T}^2)$. The main tool we employ is a standard generalization of Gronwall inequality, known as Osgood's lemma. We recall here the statement as it has been proven in \cite{cheminosgood}.
\begin{lemma}[Osgood Lemma]\label{lemma:osgood}
Let $\rho$ be a positive Borelian function, $\gamma$ a locally integrable positive function, $\mu$ a continuous increasing function and $\mathcal{M}(x) = \int_x^1\frac{dr}{\mu(r)}$. Let us assume that, for a strictly positive number $\eta$, the function $\rho$ satisfies
\begin{equation*}
    \rho(t) \leq \eta+\int_{t_0}^t \gamma(s)\mu(\rho(s))ds.
\end{equation*}
Then, we have
\begin{equation*}
    \mathcal{M}(\eta)-\mathcal{M}(\rho(t))\leq\int_{t_0}^t\gamma(s)ds.
\end{equation*}
\end{lemma}
\subsection{Rate of convergence for the velocities}
As a first step to prove \textsc{Theorem \ref{thm:rate}}, we use a technique introduced by Chemin in \cite{chemin}  to control the convergence of the velocities, employing \textsc{Lemma \ref{lemma:osgood}}.
\begin{prop}\label{thm:chemin}
Let $q^\alpha_0$ and $\omega_0$ be under the assumptions of \textsc{Theorem \ref{thm:rate}}. Let $u^\alpha_0 :=(\identity- \alpha\lapl)^{-1} k \ast q^\alpha_0$ and $u_0 :=k \ast \omega_0$, then
\begin{equation*} 
	\gamma_0^\alpha:=\|u^\alpha_0-u_0\|_{L^2} +\alpha \|\lapl u^\alpha_0\|_{L^2}\xrightarrow{\alpha \to 0}  0 \quad \text{and let} \quad \overline{\alpha}>0\quad \text{be s.t.}\quad \gamma_0^\alpha \leq \frac{1}{2},\quad \forall \alpha\leq \overline{\alpha}.
\end{equation*}
Let $(u^\alpha, q^\alpha)$ be the solution to the $\alpha$-Euler equations with initial datum $q^\alpha_0$ and let $(u, \omega)$ be its limit, which is the unique Yudovich solution to the Euler equations. Then, there exist two constants $C_1$ and $C_2$ (depending on $M:=\|\omega_0\|_{L^\infty}$) such that, if $\alpha\in (0, \overline{\alpha}]$ and $T>0$ satisfy
\begin{displaymath} 
	\alpha \leq \frac{\exp\{ 2(2- 2 \exp(C_2T))\}-\gamma_0^\alpha}{(C_1T)^2}, 
\end{displaymath}
it holds
\begin{displaymath} 
	\|u(t)-u^\alpha(t)\|_{L^2}\leq \exp\{ 2- 2 \exp(-C_2t)\}(C_1\sqrt{\alpha}T+\gamma^\alpha_0)^{\exp(-C_2t)}+ C \sqrt{\alpha}: = K(\alpha,t), \quad \forall t \leq T.
\end{displaymath} 
\end{prop}
\begin{proof}
We observe that considering the assumptions of \textsc{Theorem \ref{thm:rate}}, we have that \textsc{Theorem \ref{thm:strongconvvort}} holds true for every $p < \infty$. Therefore, we know that $q^\alpha$ converges to $\omega$ in $\mathcal{C}([0,T]; L^p(\mathbb{T}^2))$, for every $p< \infty$ and it is uniformly bounded in $L^\infty((0,T) \times \mathbb{T}^2)$. Hence, we deduce that $\omega \in L^\infty((0,T) \times \mathbb{T}^2)$, which implies that $u$ is the unique Yudovich solution to the Euler equations (cf. \cite{yudo}). Now, set
\begin{equation}
    y^\alpha := v^\alpha -u \quad \text{and} \quad z^\alpha(t) = \| y^\alpha(t,\cdot)\|_{L^2}.\label{eq:defche}
\end{equation}
Let us take the difference between the Euler equations and the $\alpha$-Euler equations and test it against $y^\alpha$. Summing the terms $\pm \int_{\mathbb{T}^2}v^\alpha\cdot \nabla u\cdot y^\alpha dx$, we obtain
\begin{multline*}
    \frac{1}{2} \totdertime\|y^\alpha\|^2=\int_{\mathbb{T}^2}(v^\alpha\cdot \nabla) u\cdot y^\alpha dx-\int_{\mathbb{T}^2}(u^\alpha \cdot \nabla) v^\alpha \cdot y^\alpha dx -\int_{\mathbb{T}^2}\sum_{j=1,2}v_j^\alpha\partial_i u_j^\alpha y_i^\alpha dx\\ - \int_{\mathbb{T}^2}(v^\alpha\cdot \nabla) u\cdot y^\alpha dx+\int_{\mathbb{T}^2}(u \cdot \nabla )u \cdot y^\alpha dx.
\end{multline*}
Using the definitions of $v^\alpha$ and $y^\alpha$, after some computations, we infer
\begin{multline*}
		\frac{1}{2} \totdertime\|y^\alpha\|^2=-\alpha\int_{\mathbb{T}^2}(\lapl u^\alpha\cdot \nabla) u\cdot y^\alpha dx-\int_{\mathbb{T}^2} (u^\alpha\cdot \nabla) y^\alpha\cdot y^\alpha dx\\-\int_{\mathbb{T}^2}\sum_{j=1,2}u_j^\alpha\partial_i u_j^\alpha y_i^\alpha dx+\alpha\int_{\mathbb{T}^2}\sum_{j=1,2}\lapl u_j^\alpha\partial_i u_j^\alpha y_i^\alpha dx - \int_{\mathbb{T}^2}( y^\alpha \cdot \nabla) u\cdot y^\alpha dx.
\end{multline*}
Owing to the incompressibility constrains, we cancel out some terms and we obtain
\begin{equation}
	\frac{1}{2} \totdertime\|y^\alpha\|^2=-\alpha\int_{\mathbb{T}^2}(\lapl u^\alpha\cdot \nabla) u\cdot y^\alpha dx+\alpha\int_{\mathbb{T}^2}\sum_{j=1,2}\lapl u_j^\alpha\partial_i u_j^\alpha y_i^\alpha dx - \int_{\mathbb{T}^2} (y^\alpha \cdot \nabla) u\cdot y^\alpha dx.
\label{eq:conticorrchemin}
\end{equation}
We first control the last term in the right hand side of \eqref{eq:conticorrchemin}
with a technique analogous to the one used in \cite{chemin}. Using Hölder inequality and \eqref{eq:calderonzyg}, we get
\begin{equation}
   \int_{\mathbb{T}^2} (y^\alpha \cdot \nabla) u\cdot y^\alpha dx\leq\|\nabla u\|_{L^p}\left(\int_{\mathbb{T}^2}|y^\alpha|^{\frac{2p}{p-1}}dx\right)^{\frac{p-1}{p}}\leq C p\left(\int_{\mathbb{T}^2}|y^\alpha|^{\frac{2p}{p-1}}dx\right)^{\frac{p-1}{p}}.\label{eq:y1}
\end{equation}
We observe that $y^\alpha=-k \ast(\omega-q^\alpha)$. By Calderon-Zygmund inequality, for $q >2$, we have
\begin{equation}
	\|y^\alpha\|_{L^\infty}\leq C\|y^\alpha\|_{W^{1,q}} \leq C\|\omega-q^\alpha\|_{L^q} \leq C \|\omega_0\|_{L^q}\leq C.\label{eq:y3}
\end{equation}
The right hand side of \eqref{eq:y1} is controlled as
\begin{equation}
    \left(\int_{\mathbb{T}^2}|y^\alpha|^{\frac{2p}{p-1}}dx\right)^{\frac{p-1}{p}}= \left(\int_{\mathbb{T}^2}|y^\alpha|^2|y^\alpha|^{\frac{p}{p-1}}dx\right)^{\frac{p-1}{p}}\leq \|y^\alpha\|_{L^2}^{2\frac{p-1}{p}} \|y^\alpha\|_{L^\infty}^{\frac{2}{p}}\leq C\|y^\alpha\|_{L^2}^{2\frac{p-1}{p}}.\label{eq:y2}
\end{equation}
We are left to control the first two terms in the right hand side of \eqref{eq:conticorrchemin}. By \textsc{Lemma \ref{lemma:iftimie}}, we deduce $\sqrt{\alpha} \|\lapl u^\alpha\|_{L^2} \leq C \|q^\alpha_0\|_{L^2}$. Exploiting this fact, we control the first two terms in the right hand side of \eqref{eq:conticorrchemin} with some applications of Hölder inequality and we deduce
\begin{multline}
  \alpha \left|\int_{\mathbb{T}^2}(\lapl u^\alpha\cdot \nabla) u\cdot y^\alpha dx\right|+\alpha\left|\int_{\mathbb{T}^2}\sum_{j=1,2}\lapl u_j^\alpha\partial_i u_j^\alpha y_i^\alpha dx\right|\\ \leq \sqrt{\alpha}(\sqrt{\alpha}\|\lapl u^\alpha\|_{L^2})(\|\nabla u\|_{L^{2}}+\|\nabla u^\alpha\|_{L^{2}})\| y^\alpha\|_{L^{\infty}}\leq C \sqrt{\alpha},\label{eq:boundterminalfa}
\end{multline}
where we have used \eqref{eq:y3}.
Finally, we substitute back into \eqref{eq:conticorrchemin} the estimates   \eqref{eq:y1}-\eqref{eq:boundterminalfa} and \eqref{eq:y2} and by definition \eqref{eq:defche}, we obtain the ordinary differential inequality
\begin{equation}
    z^\alpha(t) '\leq C p (z^\alpha(t))^{\left(1-\frac{1}{p}\right)}+C\sqrt{\alpha}, \quad \forall p\geq 2.\label{eq:odiosgood}
\end{equation}
Thanks to \eqref{eq:newcit1}-\eqref{eq:newcit2}-\eqref{eq:laplpinit}, we deduce \eqref{eq:closeenough}, namely
\begin{displaymath}
	\gamma_0^\alpha:= \|u^\alpha_0 -u_0\|_{L^2}+\alpha\|\lapl u^\alpha_0\| \rightarrow 0,\quad \text{as}\quad \alpha \to 0.
\end{displaymath}
Thus, let $\overline{\alpha}$ be such that $\gamma_0^\alpha\leq 1/2$, for every $\alpha\leq \overline{\alpha}$.
 Hereafter until the end of the proof, we consider $\alpha \in (0, \overline{\alpha}]$ and an application of triangular inequality yields
\begin{equation}
z^\alpha(0) \leq \gamma_0^\alpha\leq \frac{1} {2}.\label{eq:initzalpha}
\end{equation}
We want to apply a substitution for which we need $z^\alpha(t) \in (0,1]$. We know that $z^\alpha(0)\leq 1/2$ by \eqref{eq:initzalpha} and we define a ``modified" $z^\alpha$ as
\begin{equation*}
    z^\alpha_\delta(t) := z^\alpha(t)+\delta \quad \text{where} \enskip \delta\in \left(0, \frac{1}{2}\right]\quad \text{such that} \quad z^\alpha_\delta(0)<1\quad \text{and} \quad z^\alpha_\delta(t)>0, \enskip \forall t \geq 0.
\end{equation*}
From this definition, it is easy to check that \eqref{eq:odiosgood} still holds for $z^\alpha_\delta$, namely
\begin{equation}
    z^\alpha_\delta(t) '\leq C p (z^\alpha_\delta(t))^{\left(1-\frac{1}{p}\right)}+C\sqrt{\alpha}, \quad \forall p\geq 2.\label{eq:odiosgooddelta}
\end{equation}
We notice that \eqref{eq:odiosgooddelta} holds for any $p \geq 2$, therefore we consider $p = 2-\log(z^\alpha_\delta(t))$, which is well defined thanks to $z^\alpha_\delta(t) \in (0,1]$. We infer with some simple computations
\begin{equation*}
      (z^\alpha_\delta)'\leq C (2-\log(z^\alpha_\delta)) z^\alpha_\delta+C\sqrt{\alpha},
\end{equation*}
thanks to $(z^\alpha_\delta)^{-\frac{1}{\log(z^\alpha_\delta)}}= e^{-1}$. Integrating over time, we obtain
\begin{equation*}
 z^\alpha_\delta(t) -z^\alpha_\delta(0)\leq C_1\sqrt{\alpha}T+\int_0^tC_2 (2-\log(z^\alpha_\delta(s))) z^\alpha_\delta(s)ds, \qquad \forall t \in (0,T].
\end{equation*}
We apply \textsc{Lemma \ref{lemma:osgood}} with $\mu(x): = x(2-\log(x))$, $\mathcal{M}(x) := \log(2-\log(x))-\log2$, $\eta := C_1\sqrt{\alpha}T+z^\alpha_\delta(0)$ and $\gamma(s):= C_2$, yielding
\begin{equation*}
    -\log(2-\log(z^\alpha_\delta (t) )+\log(2-\log(C_1\sqrt{\alpha}+z^\alpha_\delta(0))) \leq C_2 t,
\end{equation*}
which implies
\begin{equation}
    z^\alpha_\delta (t) \leq \exp\{ 2- 2 \exp(-C_2t)\}(C_1\sqrt{\alpha}T+z^\alpha_\delta(0))^{\exp(-C_2t)}.\label{eq:thesischemin}
\end{equation}
We recall that we required to start with an initial datum such that $z^\alpha_\delta(0) <1$ and the inequality \eqref{eq:thesischemin} holds as long as $z^\alpha_\delta(t)\leq 1$. This statement holds true as long as we consider
\begin{equation*}
 \alpha \leq \frac{\exp\{2( 2- 2 \exp(C_2T))\}-\gamma_0^\alpha}{(C_1T)^2}.
\end{equation*}
We can pass to the limit as $\delta \to 0$ and deduce that \eqref{eq:thesischemin} still holds if we take $z^\alpha(t)$ in place of $z^\alpha_\delta(t)$. We complete the proof of the theorem through an application of triangular inequality, yielding
\begin{multline*}
    \|u-u^\alpha\|_{L^2} \leq \|u-v^\alpha\|_{L^2} +\|v^\alpha-u^\alpha\|_{L^2} \leq z^\alpha +\alpha \|\lapl u^\alpha\|_{L^2}\\
    \leq \exp\{ 2- 2 \exp(-C_2t)\}(C_1\sqrt{\alpha}T+\gamma_0^\alpha)^{\exp(-C_2t)}+ C \sqrt{\alpha},
\end{multline*}
where we have used \eqref{eq:thesischemin} and \textsc{Lemma \ref{lemma:iftimie}}.
\end{proof}
\subsection{Rate of convergence for the vorticities}
In this paragraph, we conclude the proof of \textsc{Theorem \ref{thm:rate}}. We use the ``Chemin-type" estimate given by \textsc{Proposition \ref{thm:chemin}} to control the rate of convergence of the flows. This allows us to bound the rate of convergence of the vorticities using the definition of Lagrangian solution.
\begin{prop}\label{prop:thm2second}
		Let $\omega_0$ and $q^\alpha_0$ be under the assumptions of \textsc{Theorem \ref{thm:rate}}, let $(u^\alpha, q^\alpha)$ be the solution to the $\alpha$-Euler equations with initial datum $q^\alpha_0$ and let $(u, \omega)$ be its limit, which is the unique Yudovich solution to the Euler equations.\\ Then, there exists a value $\alpha_0 = \alpha_0(T,M, \omega_0)$ and a continuous function $\psi_{\omega_0,p,M}: \mathbb{R}^+\rightarrow\mathbb{R}^+$ vanishing at zero, such that for every $\alpha \leq \alpha_0$ holds
	\begin{equation*}
		\sup_{t \in (0,T)}\|q^\alpha(t) -\omega(t)\|_{L^p}\leq C M^{1-\frac{1}{p}}\max\left\{ \psi_{\omega_0,p,M}(K(\alpha,t)),K(\alpha,t)^{\frac{\exp{(-CT)}}{2p}} \right\},
	\end{equation*}
where $C$ depends on $M$.
\end{prop}
\begin{proof}\textsc{\textbf{Step 1}: Quantitative estimate for the flows.}
Let $X^\alpha_{t,s}$ and $X_{t,s}$ be the two Lagrangian flows relative to $u^\alpha$ and $u$. Subtracting the two respective definitions of the flows, multiplying them by $X^\alpha_{t,s}-X_{t,s}$ and integrating over time, we obtain 
\begin{equation}
    \frac{|X^\alpha_{t,s}-X_{t,s}|^2}{2} = \int_t^s\left( u^\alpha(\tau,X^\alpha_{t, \tau})-u(\tau,X_{t, \tau}) \right) \left( X^\alpha_{t,\tau}-X_{t,\tau}\right)d \tau.\label{eq:citaosgood1}
\end{equation}
We proceed by summing to the right hand side $\pm u(\tau, X^\alpha_{t, \tau})$ and we infer
\begin{multline*}
   \int_t^s \left|\left( u^\alpha(\tau,X^\alpha_{t, \tau})-u(\tau,X_{t, \tau}) \right) \left( X^\alpha_{t,\tau}-X_{t,\tau}\right) \right|d \tau \leq\\ \int_t^s\left|\left( u^\alpha(\tau,X^\alpha_{t, \tau})-u(\tau,X^\alpha_{t, \tau})\right) \left( X^\alpha_{t,\tau}-X_{t,\tau}\right)\right|d \tau +\int_t^s\left| \left(u(\tau,X^\alpha_{t, \tau})-u(\tau,X^\alpha_{t, \tau}) \right) \left( X^\alpha_{t,\tau}-X_{t,\tau}\right)\right|d \tau.
\end{multline*}
Integrating over the torus, we bound the first term owing to the measure preserving property of $X^\alpha_{t,\tau}$ and using Young inequality. The second term is bounded thanks to \eqref{eq:maxine}, hence we infer
\begin{multline}
   \int_{\mathbb{T}^2}\int_t^s\left|\left( u^\alpha(\tau,X^\alpha_{t, \tau})-u(\tau,X_{t, \tau}) \right) \left( X^\alpha_{t,\tau}-X_{t,\tau}\right)\right|d \tau dx\\ \leq  \frac{1}{2}\|u^\alpha-u\|_{L^\infty_t L^2_x}|t-s|+ \int_{\mathbb{T}^2}\int_t^s\frac{|X^\alpha_{t,\tau}-X_{t,\tau}|^2}{2}d\tau dx\\ + \int_{\mathbb{T}^2}\int_t^s \frac{|X^\alpha_{t,\tau}-X_{t,\tau}|^2}{2}\left[\mathcal{M}\nabla u(\tau, \cdot)(X^\alpha_{t,\tau})+\mathcal{M}\nabla u(\tau, \cdot)(X_{t,\tau}))\right]d \tau dx.\label{eq:citastefano}
\end{multline}
 By substituting \eqref{eq:citaosgood1} into \eqref{eq:citastefano} and employing Hölder inequality on the last term, we get
 \begin{multline}
       \int_{\mathbb{T}^2}\frac{|X^\alpha_{t,s}-X_{t,s}|^2}{2} dx\leq \frac{1}{2}\|u^\alpha-u\|_{L^\infty_t L^2_x}|t-s|+ \int_{\mathbb{T}^2}\int_t^s\frac{|X^\alpha_{t,\tau}-X_{t,\tau}|^2}{2}d\tau dx\\ + 
      \|\nabla u \|_{L^p}\int_t^s\int_{\mathbb{T}^2}\left( \frac{|X^\alpha_{t,\tau}-X_{t,\tau}|^2}{2}\right)^{\frac{p-1}{p}}ds dx,
       \label{eq:cita2osgood}
 \end{multline}
where we have used that $X^\alpha_{t, \tau}$ and $X_{t, \tau}$ take values in $\mathbb{T}^2$.\\
 Let $y(t,s)=\int_{\mathbb{T}^2}\left( \frac{|X^\alpha_{t,s}-X_{t,s}|^2}{2}\right) dx$, then the equation \eqref{eq:cita2osgood} can be written as
 \begin{flalign*}
 \begin{cases}
 y(t,s) &\leq   K(\alpha,t) +\int^t_s\left(  y(t,\tau)+C p y(t,\tau)^{\frac{p-1}{p}}\right)d \tau\\
 y(t,t) &= 0.
 \end{cases}
\end{flalign*}
 Noticing the similarity with \eqref{eq:odiosgood}, we can proceed in analogous way because $y \in (0,1)$. We define $p = 2- \log(y(t,\tau))$ and applying \textsc{Lemma \ref{lemma:osgood}}, we obtain
 \begin{equation*}
     -\log(2-\log(y(t,s))+\log(2-\log(K(\alpha,t)) \leq C(t-s),
 \end{equation*}
 which implies
 \begin{equation}
    \int_{\mathbb{T}^2} |X^\alpha_{t,s}-X_{t,s}|^2  dx \leq 2 \exp\left\{2-2e^{-c(t-s)} \right\}K(\alpha,t)e^{-CT}. \label{eq:rateflow}
 \end{equation}
\textsc{\textbf{Step 2}: Quantitative estimate for the vorticities.}
Since $\omega_0 \in L^\infty(\mathbb{T}^2)$ and therefore in $L^1(\mathbb{T}^2)$, there exists a modulus of continuity in $L^1$, that is a continuous function $\psi_{\omega_0}(\cdot)$ vanishing at zero, such that
\begin{equation*}
    \|\omega_0(\cdot+h)-\omega_0(\cdot)\|_{L^1(\mathbb{T}^2)}\leq \psi_{\omega_0}(|h|), \quad \text{for any}\quad h \in \mathbb{T}^2.
\end{equation*}
Therefore for any $\varepsilon>0$, we get
\begin{align*}
    \|q^\alpha(t) -\omega(t)\|_{L^1(\mathbb{T}^2)}&= \|q^\alpha_0(X^\alpha_{t,0}) -\omega_0(X_{t,0})\|_{L^1(\mathbb{T}^2)}\\
    & \leq \int_{\dist\left(X^\alpha_{t,0}(x),X_{t,0}(x)\right)\leq \varepsilon}\left|\omega_0(X^\alpha_{t,0}(x)) -\omega_0(X_{t,0}(x))\right| dx\\&+ \int_{\dist\left(X^\alpha_{t,0}(x),X_{t,0}(x)\right)> \varepsilon}\left|\omega_0(X^\alpha_{t,0}(x)) -\omega_0(X_{t,0}(x))\right| dx\\&+\|q^\alpha_0(X^\alpha_{t,0})-\omega_0(X^\alpha_{t,0})\|_{L^1}\\
    & \leq \psi_{\omega_0}(\varepsilon)+\frac{2 \|\omega_0\|_{L^\infty(\mathbb{T}^2)}}{\varepsilon^2}  \int_{\mathbb{T}^2} |X^\alpha_{t,s}-X_{t,s}|^2dx +\|q^\alpha_0(X^\alpha_{t,0})-\omega_0(X^\alpha_{t,0})\|_{L^1}.
\end{align*}
Exploiting the fact that $X^\alpha_{t,s}$ is measure preserving and using inequality \eqref{eq:rateflow}, we infer
\begin{equation}
     \|q^\alpha(t) -\omega(t)\|_{L^1(\mathbb{T}^2)} \leq \psi_{\omega_0}(\varepsilon)+\frac{2 \|\omega_0\|_{L^\infty(\mathbb{T}^2)}}{\varepsilon^2} \left\{2-2e^{-c(t-s)} \right\}K(\alpha,t)e^{-CT} +\|q^\alpha_0-\omega_0\|_{L^1}.\label{eq:citareee}
\end{equation}
Finally, we take $\varepsilon =K(\alpha,t)e^{-\frac{CT}{4}}$ and interpolate $L^p$ between $L^1$ and $L^\infty$ to obtain the thesis.
\end{proof}
Combining \textsc{Proposition \ref{thm:chemin}} with \textsc{Proposition \ref{prop:thm2second}} completes the proof of \textsc{Theorem \ref{thm:rate}}.
Lastly, the proof of \textsc{Corollary \ref{cor:besov}} follows directly from the definition of Besov spaces.
\begin{proof}[Proof of Corollary \ref{cor:besov}] 
We prove a control on the modulus of continuity under the additional hypothesis $\omega_0 \in B^s_{p, \infty}(\mathbb{T}^2)$. A well known characterization of the Besov spaces $B^s_{p,\infty}(\mathbb{T}^2)$, for $s \in (0,1)$, is given by
\begin{equation}
B^s_{p,\infty}(\mathbb{T}^2)=\left\{ 	f(x) \in L^p(\mathbb{T}^2) \enskip \text{s.t.}\enskip \sup_{h \in \mathbb{T}^2 \setminus \{0\}} \frac{\|f(\cdot+h)-f(\cdot)\|_{L^p}}{|h|^s} < \infty\right\}. \label{eq:defbesov}
\end{equation}
Hence, the characterization \eqref{eq:defbesov} implies
\begin{equation*}
\|\omega_0 (\cdot+h)-\omega_0(\cdot)\|_{L^p} \leq C|h|^s,
\end{equation*}
thus, we can take $\psi_{\omega_0}(|h|) \simeq |h|^s $ into \eqref{eq:modulcontest} and therefore \eqref{eq:ratebesov} follows.
\end{proof}
\bibliographystyle{plain}
\bibliography{sources}
\end{document}